\documentclass[11pt,reqno,letterpaper]{amsart}

\usepackage{amssymb}
\usepackage{amsmath}
\usepackage{amsthm}
\usepackage{bbm}
\usepackage{doi}
\usepackage{graphicx}

\usepackage{tikz}
\usetikzlibrary{positioning,intersections, calc, arrows,decorations.markings}
\usepackage{tikz-3dplot}

\usepackage{bookmark}
\usepackage{hyperref}
\hypersetup{pdfstartview={FitH}}

\numberwithin{equation}{section}

\newtheorem{theorem}{Theorem}
\newtheorem{lemma}[theorem]{Lemma}
\newtheorem{corollary}[theorem]{Corollary}

\theoremstyle{definition}
\newtheorem{remark}[theorem]{Remark}

\renewcommand{\leq}{\leqslant}
\renewcommand{\geq}{\geqslant}

\newcommand{\R}{\mathbb{R}}
\newcommand{\N}{\mathbb{N}}
\newcommand{\Z}{\mathbb{Z}}
\newcommand{\T}{\mathbb{T}}

\begin{document}

\title{Inequalities in Fourier analysis on binary cubes}

\author[T. Crmari\'{c}]{Ton\'{c}i Crmari\'{c}}
\address{T. C., Department of Mathematics, Faculty of Science, University of Split, Ru\dj{}era Bo\v{s}kovi\'{c}a 33, 21000 Split, Croatia}
\email{tcrmaric@pmfst.hr}

\author[V. Kova\v{c}]{Vjekoslav Kova\v{c}}
\address{V. K., Department of Mathematics, Faculty of Science, University of Zagreb, Bijeni\v{c}ka cesta 30, 10000 Zagreb, Croatia}
\email{vjekovac@math.hr}

\author[S. Shiraki]{Shobu Shiraki}
\address{S. S., Department of Mathematics, Faculty of Science, University of Zagreb, Bijeni\v{c}ka cesta 30, 10000 Zagreb, Croatia}
\email{shobu.shiraki@math.hr}



\begin{abstract}
This paper studies two classical inequalities, namely the Hausdorff--Young inequality and equal-exponent Young's convolution inequality, for discrete functions supported in the binary cube $\{0,1\}^d\subset\mathbb{Z}^d$. We characterize the exact ranges of Lebesgue exponents in which sharp versions of these two inequalities hold, and present several immediate consequences. First, if the functions are specialized to be the indicator of some set $A\subseteq\{0,1\}^d$, then we obtain sharp upper bounds on two types of generalized additive energies of $A$, extending the works of Kane--Tao, de Dios Pont--Greenfeld--Ivanisvili--Madrid, and  one of the present authors. Second, we obtain a sharp binary variant of the Beckner--Hirschman entropic uncertainty principle, as well as a sharp lower estimate on the entropy of a sum of two independent random variables with values in $\{0,1\}^d$. Finally, the sharp binary Hausdorff--Young inequality also reveals the exact range of dimension-free estimates for the Fourier restriction to the binary cube.
\end{abstract}

\maketitle



\section{Introduction}
Sharp constants in the most classical inequalities in Fourier analysis, namely the Hausdorff--Young inequality and Young’s convolution inequality, are known since William Beckner's seminal paper \emph{Inequalities in Fourier analysis} \cite{Bec75}, which was published precisely $50$ years ago.
Our intention is to pay homage to \cite{Bec75} by studying sharp dimension-free variants of these two inequalities for functions supported in the binary cubes $\{0,1\}^d$. 
In this setting the question is no longer about the sharp constants, which end up being simply $1$, but on the optimal (enlarged) ranges of exponents for which the estimates hold. These ranges will be characterized as quite unusual curved regions depicted in Figures \ref{fig:dimension_free} and \ref{fig:Young} below.
The approach also shares some formal similarity with original Beckner's method: the results are reduced to certain two-point and four-point inequalities (see Lemmata \ref{lm:twopoint} and \ref{lm:twopointYoung} below) and the main work is spent on their technical proofs.

\subsection{The Hausdorff--Young inequality}
Let $d$ be a positive integer.
The Fourier transform
\[ \widehat{f}\colon\T^d \to\mathbb{C}, \quad \widehat{f}(\xi) := \sum_{x\in\Z^d} f(x) e^{-2\pi i x\cdot \xi} \]
of a discrete function $f\in\ell^1(\Z^d)$ satisfies the well-known Hausdorff--Young inequality
\begin{equation}\label{eq:ordinaryHY} 
\|\widehat{f}\|_{\textup{L}^q(\T^d)} \leq \|f\|_{\ell^p(\Z^d)}
\end{equation}
if and only if the exponents $p,q\in[1,\infty]$ satisfy
\[ p\leq 2 \quad\text{and}\quad \frac{1}{p} + \frac{1}{q} \geq 1; \]
see the left half of Figure \ref{fig:dimension_free}.
This range of Lebesgue exponents can be written in a somewhat unusual way,
\begin{equation}\label{eq:ordinaryrange}
\frac{1}{p} \geq 
\begin{cases} 1 - \displaystyle\frac{1}{q} & \text{for } q\in[2,\infty], \\
\displaystyle\frac{1}{2} & \text{for } q\in[1,2),
\end{cases}
\end{equation}
in order to be consistent with the later text.
On the other hand, the Fourier transform of $g\in\textup{L}^1(\T^d)$ is defined as
\[ \widehat{g}\colon\Z^d \to\mathbb{C}, \quad \widehat{g}(x) := \int_{\T^d} g(\xi) e^{-2\pi i x\cdot \xi} \,\textup{d}\xi. \]
The dual estimate to \eqref{eq:ordinaryHY} is
\begin{equation}\label{eq:ordinaryHY2} 
\|\widehat{g}\|_{\ell^{p'}(\Z^d)} \leq \|g\|_{\textup{L}^{q'}(\T^d)}, 
\end{equation}
so it holds in the exact same range of exponents \eqref{eq:ordinaryrange}.
Here and in what follows $p'$ and $q'$ respectively denote the conjugate exponents to $p$ and $q$, i.e., $1/p+1/p'=1$, $1/q+1/q'=1$.

\begin{figure}
\includegraphics[width=0.4\linewidth]{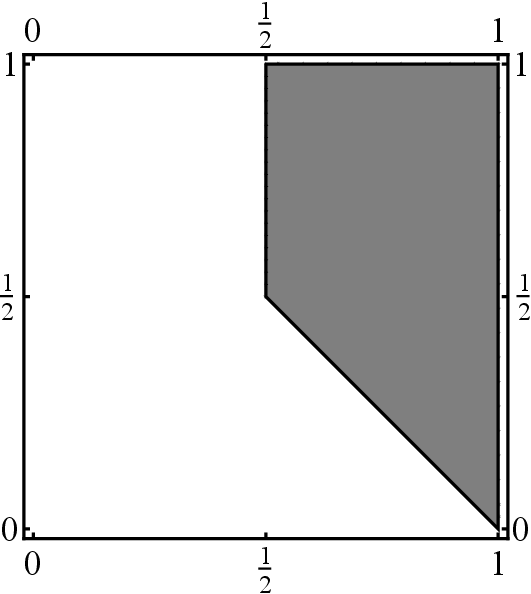}\hspace*{1cm}
\includegraphics[width=0.4\linewidth]{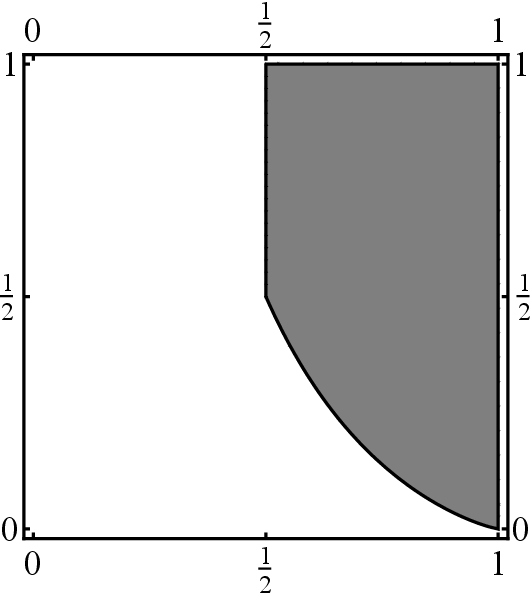}
\caption{Regions determined by points $(1/p,1/q)\in[0,1]^2$ from the exponent ranges \eqref{eq:ordinaryrange} (left) and \eqref{eq:improvedrange} (right).}
\label{fig:dimension_free}
\end{figure}

An interesting open-ended question is whether the above range of $(p,q)$ for the estimate \eqref{eq:ordinaryHY} can be improved under an additional assumption that $f$ is supported in a particular subset of $\Z^d$. Apparently the simplest and the most studied case (as will be discussed a bit later) seems to be the one of the binary cube $\{0,1\}^d$ sitting inside the lattice group $\Z^d$.
Dually, one can study inequality \eqref{eq:ordinaryHY2} when, after computing the Fourier transform of $g$, the function $\widehat{g}$ is additionally restricted to the binary cube $\{0,1\}^d$.
We are able to fully characterize the validity of the corresponding estimates.

For real numbers $a,b$ such that $\min\{a,b,a-b\}>-1$ we define the generalized binomial coefficient $\binom{a}{b}$ as
\[ \binom{a}{b} := \frac{\Gamma(a+1)}{\Gamma(b+1)\Gamma(a-b+1)}, \]
where $\Gamma$ denotes the usual gamma function.

\begin{theorem}\label{thm:main}
The following are equivalent for $p,q\in[1,\infty]$.
\begin{enumerate}
\item \label{thmit1} The estimate
\begin{equation}\label{eq:improvedHY} 
\|\widehat{f}\|_{\textup{L}^q(\T^d)} \leq \|f\|_{\ell^p(\{0,1\}^d)}
\end{equation}
holds for every $d\in\N$ and every $f\colon\Z^d\to\mathbb{C}$ supported in the binary cube $\{0,1\}^d$.
\item \label{thmit2} The estimate
\begin{equation}\label{eq:improvedHY2} 
\big\|\widehat{g}\big|_{\{0,1\}^d}\big\|_{\ell^{p'}(\{0,1\}^d)} \leq \|g\|_{\textup{L}^{q'}(\T^d)}
\end{equation}
holds for every $d\in\N$ and every $g\in\textup{L}^{q'}(\T^d)$.
\item \label{thmit3} The Lebesgue exponents $p,q$ satisfy
\begin{equation}\label{eq:improvedrange}
\frac{1}{p} \geq 
\begin{cases} 1 & \text{for } q=\infty, \\
\displaystyle\frac{1}{q}\log_2\binom{q}{q/2} & \text{for } q\in[2,\infty), \\[3mm]
\displaystyle\frac{1}{2} & \text{for } q\in[1,2).
\end{cases}
\end{equation}
\end{enumerate}
\end{theorem}

The optimal range \eqref{eq:improvedrange} from Theorem \ref{thm:main} is depicted in the right half of Figure \ref{fig:dimension_free}.
Conditions \eqref{thmit1} and \eqref{thmit2} are clearly equivalent by duality, while substantial work is required to show that they are also equivalent to Condition \eqref{thmit3}.

The particular case $q=4$ of Theorem \ref{thm:main} is an easy reformulation of the result by Kane and Tao \cite[Theorem 7 \& Remark 9]{KT17}. When $q$ is an even integer, the problem was studied by de Dios Pont, Greenfeld, Ivanisvili, and Madrid \cite[Theorem 3, \S 4.1 \& \S 4.2]{DGIM21}. The actual verification of \eqref{eq:improvedHY} in the optimal range of exponents $p$ for a given fixed $q=2k$, $k\in\N$, was done in \cite{DGIM21}, but only for $k\leq 10$, with a remark that it extends to $k\leq 100$ with an assistance of a computer. The cases $q=2k$ for all positive integers $k$ were subsequently completed by one of the present authors \cite[Corollary 2]{Kov23}. In the aforementioned papers the emphasis is often put on nonnegative functions $f$, or even just indicator functions $f=\mathbbm{1}_A$ of sets $A\subseteq\{0,1\}^d$, since the problem was motivated from additive combinatorics.

Speaking of combinatorial applications, for every real number $\kappa\in[1,\infty)$ we define a generalized additive energy $E_\kappa$ of a finite set $A\subset\Z^d$ as
\[ E_{\kappa}(A) := \|\widehat{\mathbbm{1}_A}\|_{\textup{L}^{2\kappa}(\T^d)}^{2\kappa} = \int_{\T^d} \Big|\sum_{x\in A} e^{-2\pi i x\cdot\xi}\Big|^{2\kappa} \,\textup{d}\xi. \]
The case $\kappa=1$ leads to merely just $E_1(A)=|A|$, i.e., the size of $A$, while the case $\kappa=2$ simplifies to the concept of the usual additive energy \cite{TV06}.
In the case of a general positive integer $\kappa$ the definition of $E_\kappa$ can be rewritten as
\begin{align*} 
E_{\kappa}(A) & = \|\underbrace{\mathbbm{1}_A\ast\mathbbm{1}_A\ast\cdots\ast\mathbbm{1}_A}_{\kappa}\|_{\ell^2(\Z^d)}^2 \\
& = \big| \{ (a_1,a_2,\ldots,a_{2\kappa-1},a_{2\kappa})\in A^{2\kappa} \,:\, a_1+\cdots+a_\kappa = a_{\kappa+1}+\cdots+a_{2\kappa} \} \big|
\end{align*}
and it was systematically studied in \cite{SS13,Shk14}.
The idea to investigate $E_\kappa$ for non-integer values of $\kappa$ has in fact already been hinted on in \cite[Lemma 16 \& Remark 17]{DGIM21}. 
An immediate consequence of our theorem is a sharp bound for the $E_\kappa$ energy of $A$ in terms of its size, for subsets $A$ of the binary cubes.

\begin{corollary}\label{cor:additive}
For every $\kappa\in[1,\infty)$, $d\in\N$, and $A\subseteq\{0,1\}^d$ one has
\begin{equation}\label{eq:additiver}
E_{\kappa}(A) \leq |A|^{r},
\end{equation}
where
\begin{equation}\label{eq:defofr}
r = r(\kappa) := \log_2\binom{2\kappa}{\kappa}.
\end{equation}
Besides trivial cases $|A|\leq 1$, the equality is also attained for some sets with $|A|>1$, e.g., if $A=\{0,1\}^d$.
\end{corollary}

The last remark in Corollary \ref{cor:additive} makes the exponent \eqref{eq:defofr} optimal. 
Particular case $\kappa=2$ of \eqref{eq:additiver} is precisely \cite[Theorem 7]{KT17}, while all integer values of $\kappa$ were jointly covered by \cite[Theorem 3]{DGIM21} and \cite[Corollary 2]{Kov23}.
Even though here we are able to handle all real $\kappa\geq1$, the present proof of \eqref{eq:additiver} is arguably also the most elegant one, even if still quite nontrivial and partially computer assisted.

Additional motivation for estimate \eqref{eq:improvedHY2} comes from its interpretation as a discrete di\-men\-sion-free Fourier restriction estimate, somewhat in the spirit of the recent work of Oliveira e Silva and Wr\'{o}bel \cite{OW24} for the spheres in $\R^d$.

\begin{corollary}\label{cor:restriction}
Let $p,q\in[1,\infty]$.
There exists a constant $C_{p,q}\in(0,\infty)$ such that for every $d\in\N$ and every $g\in\textup{L}^{q'}(\T^d)$ one has
\begin{equation}\label{eq:improvedHY3} 
\big\|\widehat{g}\big|_{\{0,1\}^d}\big\|_{\ell^{p'}(\{0,1\}^d)} \leq C_{p,q} \|g\|_{\textup{L}^{q'}(\T^d)}
\end{equation}
if and only if $p$ and $q$ satisfy \eqref{eq:improvedrange}.
\end{corollary}

Namely, the usual tensor-power trick will show that \eqref{eq:improvedHY2} and \eqref{eq:improvedHY3} are equivalent, i.e., the constant $C_{p,q}$ is forced to be $1$ if it does not depend on the dimension $d$.

The main ingredient in the proof of Theorem \ref{thm:main} is the following ``two-point inequality.''

\begin{lemma}\label{lm:twopoint}
If $\alpha,\beta\in[0,\infty)$, $p\in(1,2)$, $q\in(2,\infty)$ satisfy \eqref{eq:improvedrange}, then
\begin{equation}\label{eq:twopointineq} 
\Big( \int_0^1 |\alpha e^{2\pi i t} + \beta|^q \,\textup{d}t \Big)^{1/q} \leq (\alpha^p + \beta^p)^{1/p}.
\end{equation}
\end{lemma}

Interestingly, the proof of Lemma \ref{lm:twopoint} will rely on an ordinary differential equation (reminiscent of the Legendre ODE), which is easy to formulate in Lemma \ref{lm:ode} below, but also practically unverifiable without an assistance of a computer.

In the particular cases $q=2k$, $k\in\N$, the integral on the left-hand side of \eqref{eq:twopointineq} transforms as
\[ \int_0^1 |(\alpha e^{2\pi i t} + \beta)^k|^2 \,\textup{d}t 
= \int_0^1 \Big| \sum_{j=0}^{k} \binom{k}{j} \alpha^{k-j} \beta^{j} e^{2\pi i (k-j) t} \Big|^2 \,\textup{d}t 
= \sum_{j=0}^{k} \binom{k}{j}^2 \alpha^{2(k-j)} \beta^{2j}. \]
In the endpoint case
\[ \frac{1}{p} = \frac{1}{2k}\log_2\binom{2k}{k} \]
substitutions 
\[ a=\alpha^p, \quad b=\beta^p, \quad r=\frac{2k}{p} \]
transform \eqref{eq:twopointineq} into
\begin{equation}\label{eq:twopointineq0} 
\sum_{j=0}^{k} \binom{k}{j}^2 a^{\frac rk (k-j)} b^{\frac rkj} \leq (a+b)^r
\end{equation}
for $a,b\in[0,\infty)$ and $r=\log_2\binom{2k}{k}$.
Inequality \eqref{eq:twopointineq0} was established for $k=2$ in \cite[Lemma 8]{KT17}, for $k\leq10$ in \cite[Lemma 5]{DGIM21}, and for a general $k\in\N$ in \cite[Theorem 1]{Kov23}.

In particular, the range $2<q<4$ in Lemma \ref{lm:twopoint} and Theorem \ref{thm:main} has previously been completely unexplored, while we are now in position to describe an attractive information-theoretic application obtained by merely studying the behavior of \eqref{eq:improvedHY} as $q\to2+$.

Beckner's sharp Hausdorff--Young inequality on $\R$ from \cite{Bec75} implied a sharp uncertainty principle for the Shannon entropy on $\R$, previously conjectured by Hirschman \cite{Hir57}.
In analogy with that, we formulate a similar entropic uncertainty principle on $\T^d\times\Z^d$.
Suppose that a finitely supported function $f\colon\Z^d\to\mathbb{C}$ is normalized as
\begin{equation}\label{eq:normalization}
\sum_{x\in\Z^d} |f(x)|^2 = 1.
\end{equation}
Then $|f|^2$ can be thought of as a probability mass function on $\Z^d$ and its Shannon entropy (measured in bits) is defined as
\[ \textup{H}_{\Z^d}(|f|^2) := - \sum_{x\in\Z^d} |f(x)|^2 \log_2 |f(x)|^2, \]
where $0\log_2 0$ is interpreted as $0$.
Note that $|f(x)|\leq 1$ implies
\[ \textup{H}_{\Z^d}(|f|^2) \geq 0. \]
By Plancherel's identity then we also have
\[ \int_{\T^d} |\widehat{f}(\xi)|^2 \,\textup{d}\xi = 1, \]
so $|\widehat{f}|^2$ can be thought of as a probability density function on $\T^d$. Its Shannon entropy is naturally defined as
\[ \textup{H}_{\T^d}(|\widehat{f}|^2) := - \int_{\T^d} |\widehat{f}(\xi)|^2 \log_2 |\widehat{f}(\xi)|^2 \,\textup{d}\xi, \]
but $\textup{H}_{\T^d}(|\widehat{f}|^2)$ no longer needs to be nonnegative.
The ordinary Hausdorff--Young inequality \eqref{eq:ordinaryHY} implies the entropic uncertainty bound\footnote{The proof is an adaptation of the analogous claim on the real line \cite{Hir57,Bec75}, but it can also be read between the lines of our proof presented in Section \ref{sec:entropy}.}
\begin{equation}\label{eq:uncertainty0}
\textup{H}_{\T^d}(|\widehat{f}|^2) + \textup{H}_{\Z^d}(|f|^2) \geq 0.
\end{equation}
We are able to give an improvement for functions $f$ supported in $\{0,1\}^d$.

\begin{corollary}\label{cor:entropy}
For every $d\in\N$ and every function $f\colon\Z^d\to\mathbb{C}$ supported in the binary cube $\{0,1\}^d$ and satisfying $\|f\|_{\ell^2(\{0,1\}^d)}=1$ one has
\begin{equation}\label{eq:uncertainty}
\textup{H}_{\T^d}(|\widehat{f}|^2) + \Big(\underbrace{\frac{1}{\ln 2}-1}_{0.442695\ldots}\Big) \underbrace{\textup{H}_{\Z^d}(|f|^2)}_{\geq 0} \geq 0
\end{equation}
and the equality can be attained even in the cases when $\textup{H}_{\Z^d}(|f|^2)\neq0$, e.g., when $f$ is constant on $\{0,1\}^d$.
\end{corollary}

Estimate \eqref{eq:uncertainty} is clearly a refinement of \eqref{eq:uncertainty0} for functions supported in $\{0,1\}^d$ and it is sharp due to the second claim of the corollary.
Inequality \eqref{eq:uncertainty} can also be rewritten as
\[ \Big(1-\ln\frac{e}{2}\Big)\textup{H}_{\T^d}(|\widehat{f}|^2) + \Big(\ln\frac{e}{2}\Big) \textup{H}_{\Z^d}(|f|^2) \geq 0. \]
It is interesting to notice that the constant $\ln(e/2)$ also appears in the sharp entropic uncertainty principle on the real line:
\[ \textup{H}_{\R}(|\widehat{f}|^2) + \textup{H}_{\R}(|f|^2) \geq \ln\frac{e}{2}. \]
Here we do not define the entropy $\textup{H}_{\R}$ of probability densities on $\R$, as these are not of our current interest, and we rather just refer to \cite{Hir57,Bec75} for details.

\subsection{Young's convolution inequality}
The convolution of discrete functions $f,g\colon\Z^d\to[0,\infty)$ is defined as
\[ f\ast g\colon\Z^d\to[0,\infty],\quad (f\ast g)(x) := \sum_{y\in\Z^d} f(x-y) g(y). \]
Once again we restrict our attention to functions $f$ and $g$ supported in the binary cube $\{0,1\}^d$.
Using a similar approach as in the proof of Theorem \ref{thm:main}, we will also refine the ``diagonal'' special case of Young's convolution inequality,
\[ \|f\ast g\|_{\ell^q(\Z^d)} \leq \|f\|_{\ell^p(\Z^d)} \|g\|_{\ell^p(\Z^d)}, \]
which is known to hold whenever the exponents $p,q\in[1,\infty]$ satisfy 
\begin{equation}\label{eq:rangeYoung1}
2/p\geq 1+1/q;
\end{equation}
see the left half of Figure \ref{fig:Young}.

\begin{figure}
\includegraphics[width=0.4\linewidth]{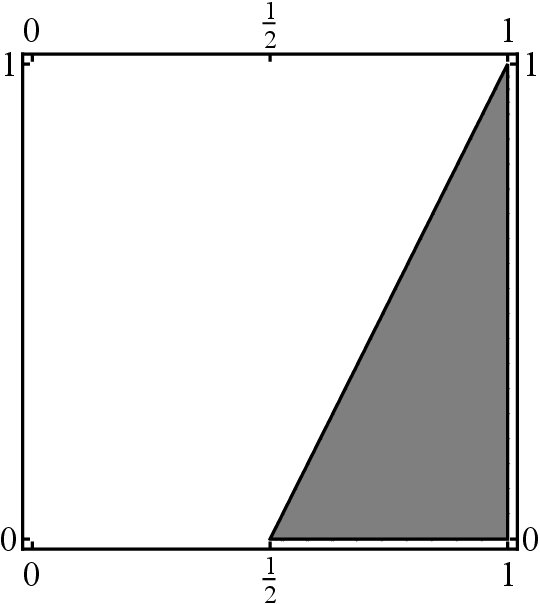}\hspace*{1cm}
\includegraphics[width=0.4\linewidth]{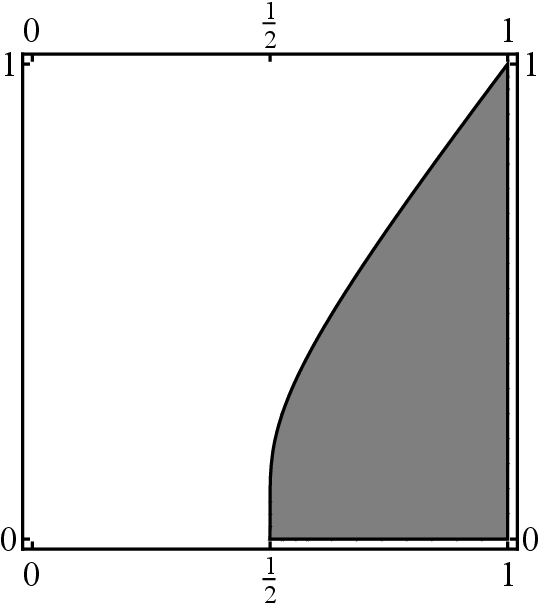}
\caption{Regions determined by points $(1/p,1/q)\in[0,1]^2$ from the exponent ranges \eqref{eq:rangeYoung1} (left) and \eqref{eq:rangeYoung2} (right).}
\label{fig:Young}
\end{figure}

\begin{theorem}\label{thm:Young}
The following are equivalent for $p,q\in[1,\infty]$.
\begin{enumerate}
\item The estimate
\begin{equation}\label{eq:improvedYoung} 
\|f\ast g\|_{\ell^q(\Z^d)} \leq \|f\|_{\ell^p(\{0,1\}^d)} \|g\|_{\ell^p(\{0,1\}^d)}
\end{equation}
holds for every $d\in\N$ and every $f,g\colon\Z^d\to\mathbb{C}$ supported in the binary cube $\{0,1\}^d$.
\item The Lebesgue exponents $p,q$ satisfy
\begin{equation}\label{eq:rangeYoung2}
\frac{1}{p} \geq 
\begin{cases} 
\displaystyle\frac{1}{2q}\log_2(2^q+2) & \text{for } q\in[1,\infty), \\[3mm]
\displaystyle\frac{1}{2} & \text{for } q=\infty.
\end{cases}
\end{equation}
\end{enumerate}
\end{theorem}

The particular case $q=2$, $p=4/\log_2 6$ of estimate \eqref{eq:improvedYoung} has been shown and commented on in \cite[Remark 9]{KT17}.

For a real number $\kappa\in[1,\infty)$ and a finite set $A\subset\Z^d$ one can define another generalization of the additive energy of $A$ as
\[ \widetilde{E}_{\kappa}(A) := \big\|\mathbbm{1}_A \ast\widetilde{\mathbbm{1}}_A\big\|_{\ell^\kappa(\Z^d)}^{\kappa}, \]
where $\widetilde{f}$ denotes the reflection of $f$, i.e., $\widetilde{f}(x):=f(-x)$ for $x\in\Z^d$.
In the particular case when $\kappa$ is a positive integer this notion was studied in \cite{SS13,Shk14,DGIM21} and it possesses a more direct combinatorial interpretation:
\[ \widetilde{E}_{\kappa}(A) = \big| \{ (a_1,a_2,\ldots,a_{2\kappa-1},a_{2\kappa})\in A^{2\kappa} \,:\, a_1-a_2 = a_3-a_4 = \cdots = a_{2\kappa-1}-a_{2\kappa} \} \big|. \]
The following consequence of Theorem \ref{thm:Young} generalizes \cite[Theorem 1]{DGIM21}, which was concerned with integer parameters $\kappa$ only.

\begin{corollary}
For every $\kappa\in[1,\infty)$, $d\in\N$, and $A\subseteq\{0,1\}^d$ one has
\[ \widetilde{E}_{\kappa}(A) \leq |A|^{\widetilde{r}}, \]
where
\[ \widetilde{r} = \widetilde{r}(\kappa) := \log_2(2^\kappa+2). \]
Besides trivial cases $|A|\leq 1$, the equality is also attained for some sets with $|A|>1$, e.g., if $A=\{0,1\}^d$.
\end{corollary}

One simply takes $q=\kappa$, $p=2\kappa/\log_2(2^\kappa+2)$ and applies\footnote{Note that $\mathbbm{1}_A$ being supported in $\{0,1\}^d$ makes $\widetilde{\mathbbm{1}}_A$ supported in $\{-1,0\}^d$. However, translation by the vector $(1,\ldots,1)$ places $\widetilde{\mathbbm{1}}_A$ back over $\{0,1\}^d$, and the $\ell^p$-norms do not see a difference.} \eqref{eq:improvedYoung} to $f=\mathbbm{1}_A$ and $g=\widetilde{\mathbbm{1}}_A$. 

Let us remark that we were not able to characterize completely validity of general Young's inequality
\begin{equation}\label{eq:toogeneralYoung}
\|f\ast g\|_{\ell^r(\Z^d)} \leq \|f\|_{\ell^p(\{0,1\}^d)} \|g\|_{\ell^q(\{0,1\}^d)} 
\end{equation}
for three exponents $p,q,r\in[1,\infty]$. Serious complications arise since $f=g=\mathbbm{1}_{\{0,1\}^d}$ is not always an ``extremal case,'' see Remark \ref{rem:notamax}.

The key ingredient in the proof of Theorem~\ref{thm:Young} is the following ``four-point inequality.''

\begin{lemma}\label{lm:twopointYoung}
Let $p\in(1,2)$ and $q\in(1,\infty)$ satisfy \eqref{eq:rangeYoung2}.
Then
\begin{equation}\label{eq:twopointineqYoung}
\big( (\alpha_0 \beta_0)^q+(\alpha_0 \beta_1 + \alpha_1 \beta_0)^q + (\alpha_1 \beta_1)^q \big)^{1/q}
\leq 
(\alpha_0^p + \alpha_1^p)^{1/p} (\beta_0^p + \beta_1^p)^{1/p}
\end{equation}
for every $\alpha_0,\alpha_1,\beta_0,\beta_1\in[0,\infty)$.
\end{lemma}

Inequality \eqref{eq:improvedYoung}, just like \eqref{eq:improvedHY}, also has an interesting application in information theory, perhaps one of even more fundamental nature. If $f\colon\Z^d\to[0,\infty)$ is a finitely supported function satisfying $\|f\|_{\ell^1(\Z^d)}=1$, then it can itself be thought of as a probability mass function on $\Z^d$ and its Shannon entropy is now 
\[ \textup{H}_{\Z^d}(f) := - \sum_{x\in\Z^d} f(x) \log_2 f(x). \]
If $g\colon\Z^d\to[0,\infty)$ is another function with finite support and normalized as $\|g\|_{\ell^1(\Z^d)}=1$, then both $g$ and $f\ast g$ can also be interpreted as probability mass functions on $\Z^d$. 
Jensen's inequality and concavity of the function $x\mapsto -x\log_2 x$ easily imply
\[ \textup{H}_{\Z^d}(f\ast g) \geq \max\{\textup{H}_{\Z^d}(f),\textup{H}_{\Z^d}(g)\}. \]
Averaging we immediately get a more symmetric version,
\begin{equation}\label{eq:entropy12}
\textup{H}_{\Z^d}(f\ast g) \geq \frac{1}{2} \big( \textup{H}_{\Z^d}(f) +\textup{H}_{\Z^d}(g) \big).
\end{equation}
The constant $1/2$ in \eqref{eq:entropy12} is still the best possible: if $f$ and $g$ correspond to the symmetric binomial distribution $B(n,1/2)$, then $f\ast g$ corresponds to $B(2n,1/2)$, so $\textup{H}_{\Z^d}(f)$, $\textup{H}_{\Z^d}(g)$, and $\textup{H}_{\Z^d}(f\ast g)$ are all asymptotically equal to $(1/2)\log_2 n$; see for instance \cite{CW79}.
We are able to make an improvement over \eqref{eq:entropy12} under an additional constraint: that $f$ and $g$ are supported in the binary cube.

\begin{corollary}\label{cor:entropy2}
For every $d\in\N$ and every $f,g\colon\Z^d\to[0,\infty)$ supported in $\{0,1\}^d$ and satisfying $\|f\|_{\ell^1(\{0,1\}^d)}=\|g\|_{\ell^1(\{0,1\}^d)}=1$ one has
\begin{equation}\label{eq:entropy34}
\textup{H}_{\Z^d}(f\ast g) \geq \frac{3}{4} \big( \textup{H}_{\Z^d}(f) +\textup{H}_{\Z^d}(g) \big).
\end{equation}
The equality is attained when $f=g=\mathbbm{1}_{\{0,1\}^d}$, which makes the constant $3/4$ sharp.
\end{corollary}

Inequality \eqref{eq:entropy34} has the following probabilistic formulation. Let $X$ and $Y$ be two independent random variables (i.e., random elements) taking values in the binary cube $\{0,1\}^d\subset\Z^d$. Then
\[ \textup{H}_{\Z^d}(X+Y) \geq \frac{3}{4} \big( \textup{H}_{\Z^d}(X) +\textup{H}_{\Z^d}(Y) \big) \]
and this inequality is sharp. Namely, if $f$ and $g$ are respectively the laws of independent variables $X$ and $Y$, then $f\ast g$ is the law of their sum $X+Y$, so that the claim follows from \eqref{eq:entropy34}.

Corollary \ref{cor:entropy2} will be an immediate consequence of the sharp binary Young estimate \eqref{eq:improvedYoung} taken to the limit as $q\to1^+$ and $p\to1^+$. One can already notice that the upper boundary of the region on the left part of Figure \ref{fig:Young} has slope $2$, while the upper boundary on the right half of the same figure has slope $4/3$ at the point $(1,1)$; leading to the desired improvement.

\subsection{A comment on larger cubes}
Some of the aforementioned work has been extended to larger discrete cubes 
\[ \{0,1,\ldots,n-1\}^d; \] 
see \cite{DGIM21,Sha24,BCK24}.
We do not pursue this line of investigation here, since no definite sharp inequalities are known when $n\geq 3$.
For instance, the optimal exponent $p$ in the triadic Hausdorff--Young inequality mapping into $\textup{L}^4(\T^d)$, namely
\[ \|\widehat{f}\|_{\textup{L}^4(\T^d)} \leq \|f\|_{\ell^p(\{0,1,2\}^d)}, \]
can be evaluated numerically in Mathematica \cite{Mathematica}. It turns out to be a somewhat arbitrary number,
\[ p = 1.4702039297\ldots, \]
not recognized in any numerical databases; compare with a comment in \cite[Section 1]{BCK24}.
On the other hand, there seems to be ongoing interest in estimates for functions supported in the binary cube $\{0,1\}^d$; see the recent papers \cite{G22,BIKM24}.


\section{Justification of Figure \ref{fig:dimension_free}}
In order to justify the depiction of the range of exponents \eqref{eq:improvedrange} in Figure \ref{fig:dimension_free} we study the behavior of the endpoint assignment 
\[ [2,\infty)\to(0,\infty), \quad q\mapsto p=p(q) \] 
given by
\begin{equation}\label{eq:pqequality}
p = \frac{q}{\log_2\binom{q}{q/2}}
= \frac{q}{\log_2\frac{\Gamma(q+1)}{\Gamma(q/2+1)^2}}.
\end{equation}

First, we claim that $q\mapsto p$ is a strictly decreasing function. Let us show that
\[ q \mapsto \frac{1}{q} \log_2 \frac{\Gamma(q+1)}{\Gamma(q/2+1)^2} \]
is strictly increasing on the whole interval $(0,\infty)$.
We can write it as $q\mapsto \varphi(q)/(q\ln 2)$, where
\[ \varphi\colon[0,\infty)\to\R, \quad \varphi(q) := \log_2 \Gamma(q+1) - 2 \log_2 \Gamma\Big(\frac{q}{2}+1\Big), \]
so it is enough to verify
\[ 0 < \frac{\textup{d}}{\textup{d}q} \frac{\varphi(q)}{q}
= \frac{1}{q} \Big(\varphi'(q)-\frac{\varphi(q)}{q}\Big). \]
Finally, it is sufficient to see that $\varphi''>0$, since then for every $q>0$ the mean value theorem gives a $\xi\in(0,q)$ such that
\[ \frac{\varphi(q)}{q} = \frac{\varphi(q)-\varphi(0)}{q-0} = \varphi'(\xi) < \varphi'(q). \]
Now, the second derivative of $\varphi$ can be expressed in terms of the trigamma function \cite[\S 5.15]{NIST},
\[ \psi'(z) := \Big(\frac{\textup{d}}{\textup{d}z}\Big)^2 \ln \Gamma(z) \]
as
\[ \varphi''(t) = \frac{1}{\ln 2}\psi'(t+1) - \frac{1}{2\ln 2}\psi'\Big(\frac{t}{2}+1\Big), \] 
and so the multiplication theorem \cite[Formula 5.15.7]{NIST} specialized as
\[ 4 \psi'(2z) = \psi'(z) + \psi'\Big(z+\frac{1}{2}\Big) \]
gives
\[ \varphi''(t) = \frac{1}{4\ln 2} \bigg( \psi'\Big(\frac{t}{2}+\frac{1}{2}\Big) - \psi'\Big(\frac{t}{2}+1\Big) \bigg). \]
This is positive for every $t>0$ because $\psi'$ is strictly decreasing on $(0,\infty)$; see for instance \cite[Formula 5.15.1]{NIST}.

Since $2$ is mapped to
\[ \frac{2}{\log_2 \binom{2}{1}} = 2, \]
the previous monotonicity observation easily implies $p\in(1,2]$.

Next, let us see that the same part of the boundary of the exponent range \eqref{eq:improvedrange}, apart from the endpoint, belongs to the sub-H\"{o}lder regime, i.e.
\begin{equation}\label{eq:subduality}
\frac{1}{p} + \frac{1}{q} < 1.
\end{equation}
To see that this holds for $q\in(2,\infty)$ and $p$ as in \eqref{eq:pqequality}, we rewrite the inequality as $\phi(q)>0$ for
\[ \phi\colon[2,\infty)\to\R, \quad \phi(t) := t - 1 - \varphi(t). \]
Since $\phi''=-\varphi''<0$, the function $\phi$ is strictly concave on $[2,\infty)$, so the claim will follow from
\[ \phi(2)=0, \quad \lim_{t\to\infty} \phi(t) = \infty. \]
To verify this, from logarithmic Stirling's formula \cite[Formula 5.11.1]{NIST} we get
\[ \log_2 \Gamma(z+1) = z\log_2 z - z \log_2 e + \frac{1}{2}\log_2 z + O(1), \quad\text{as } \R\ni z\to\infty, \]
so 
\[ \phi(t) = \frac{1}{2} \log_2 t + O(1),\quad\text{as } t\to\infty. \]

In particular, we now also conclude that the region in the right half of Figure \ref{fig:dimension_free} is strictly larger than the one in the left half of the same figure.

The above properties of the range \eqref{eq:improvedrange} are all we will need in the proofs below. Once we establish Theorem \ref{thm:main}, interpolation will also imply that the region of exponents is convex in the $(1/p,1/q)$-plane.


\section{Proof of Theorem \ref{thm:main} assuming Lemma \ref{lm:twopoint}}
\label{sec:proofthm}

\subsection{Necessity of the condition on \texorpdfstring{$p,q$}{p,q}}
We only need to assume that \eqref{eq:improvedHY} holds in dimension $d=1$, for every function $f$ supported in $\{0,1\}$.

First, by choosing
\[ f(0)=f(1)=1 \]
one obtains
\begin{equation}\label{eq:auxnecess}
\Big( \int_0^1 |1+e^{-2\pi i t}|^q \,\textup{d}t \Big)^{1/q} \leq 2^{1/p}
\end{equation}
whenever $q<\infty$.
The left hand side simplifies as
\[ \Big( \frac{2^{q+1}}{\pi} \int_0^{\pi/2} \cos^q t \,\textup{d}t \Big)^{1/q} \]
and then formula
\[ \int_0^{\pi/2} \cos^q t \,\textup{d}t = \frac{\sqrt{\pi}\,\Gamma((q+1)/2)}{2\,\Gamma(q/2+1)} \]
from \cite[Formula 5.12.2]{NIST} and the gamma-function duplication formula \cite[Formula 5.5.5]{NIST}
\[ \Gamma(q+1) = \frac{2^q}{\sqrt{\pi}} \Gamma\Big(\frac{q+1}{2}\Big) \Gamma\Big(\frac{q}{2}+1\Big) \]
transform \eqref{eq:auxnecess} into
\begin{equation}\label{eq:auxnecess1}
\frac{1}{p} \geq \frac{1}{q}\log_2\frac{\Gamma(q+1)}{\Gamma(q/2+1)^2}.
\end{equation}

The same choice of $f$ works for $q=\infty$ and leads to $2\leq 2^{1/p}$, so
\begin{equation}\label{eq:auxnecess2}
q=\infty \ \implies\ p\leq 1.
\end{equation}

Finally, let us also substitute
\[ f(0)=1, \quad f(1)=\varepsilon, \quad \varepsilon>0 \]
into \eqref{eq:improvedHY} and observe
\[ |1+\varepsilon e^{-2\pi i t}|^q = 1 + q\varepsilon \cos 2\pi t + \frac{q}{2} \varepsilon^2 \big(1+(q-2)\cos^2 2\pi t\big) + o(\varepsilon^2) \]
as $\varepsilon\to0^+$, uniformly in $t\in\R$.
Thus
\[ \Big( \int_0^1 |1+\varepsilon e^{-2\pi i t}|^q \,\textup{d}t \Big)^{1/q} 
= \Big( 1 + \frac{q^2}{4} \varepsilon^2 + o(\varepsilon^2) \Big)^{1/q}
= 1 + \frac{q}{4} \varepsilon^2 + o(\varepsilon^2) \]
for $q<\infty$ and
\[ (1+\varepsilon^p)^{1/p} = 1 + \frac{1}{p} \varepsilon^p + o(\varepsilon^p) \]
for $p<\infty$, both as $\varepsilon\to0^+$.
That way \eqref{eq:improvedHY} implies 
\[ \frac{q}{4} \varepsilon^2 + o(\varepsilon^2) \leq \frac{1}{p} \varepsilon^p + o(\varepsilon^p) \quad\text{as } \varepsilon\to0^+, \]
which is only possible when 
\begin{equation}\label{eq:auxnecess3}
p\leq 2.
\end{equation}

Note that \eqref{eq:auxnecess1}, \eqref{eq:auxnecess2}, and \eqref{eq:auxnecess3} together complete the proof of \eqref{eq:improvedrange} for all $q\in[1,\infty]$.

\subsection{Sufficiency of the condition on \texorpdfstring{$p,q$}{p,q}}
Suppose that the coefficients $p$ and $q$ satisfy \eqref{eq:improvedrange}. We can freely assume that $q\in(2,\infty)$ and $p\in(1,2)$, since for other exponents there is nothing to prove beyond the ordinary Hausdorff--Young inequality.  

Any two complex numbers $a,b$ can be written as
\[ a=\alpha e^{2\pi i \theta}, \quad b=\beta e^{2\pi i \vartheta} \]
for some $\alpha,\beta\in[0,\infty)$ and $\theta,\vartheta\in\R$.
Then 
\[ \int_0^1 |a + b e^{-2\pi i t}|^q \,\textup{d}t 
= \int_0^1 |\alpha e^{2\pi i (t+\theta-\vartheta)} + \beta|^q \,\textup{d}t
= \int_0^1 |\alpha e^{2\pi i t} + \beta|^q \,\textup{d}t, \]
so Lemma \ref{lm:twopoint} gives
\begin{equation}\label{eq:twopointineqcon} 
\Big( \int_0^1 |a + b e^{-2\pi i t}|^q \,\textup{d}t \Big)^{1/q} \leq (|a|^p + |b|^p)^{1/p} \quad \text{for } a,b\in\mathbb{C}.
\end{equation}

We are proving \eqref{eq:improvedHY} by the mathematical induction on the dimension $d\in\N$ and the basis $d=1$ becomes precisely \eqref{eq:twopointineqcon} after the substitution $a=f(0)$, $b=f(1)$.
For the induction step take some integer $d\geq 2$ and a function $f\colon\Z^d\to\mathbb{C}$ supported in $\{0,1\}^d$.
Define $f_0,f_1\colon\Z^{d-1}\to\mathbb{C}$ by
\[ f_0(x) := f(\underbrace{x}_{\in\Z^{d-1}},\underbrace{0}_{\in\Z}), \quad f_1(x) := f(\underbrace{x}_{\in\Z^{d-1}},\underbrace{1}_{\in\Z}) \]
for every $x\in\Z^{d-1}$, so that 
\[ \widehat{f}(\xi,t) = \widehat{f_0}(\xi) + \widehat{f_1}(\xi) e^{-2\pi i t} \]
for $(\xi,t)\in\T^{d-1}\times\T$.
Then \eqref{eq:twopointineqcon} and Minkowski's inequality (thanks to $p<q$) give
\begin{align*}
\|\widehat{f}\|_{\textup{L}^q(\T^d)}
& = \bigg( \int_{\T^{d-1}} \int_{\T} \big| \widehat{f_0}(\xi) + \widehat{f_1}(\xi) e^{-2\pi i t} \big|^q \,\textup{d}t \,\textup{d}\xi \bigg)^{1/q} \\
& \!\!\stackrel{\eqref{eq:twopointineqcon}}{\leq}
\bigg( \int_{\T^{d-1}} \bigg( |\widehat{f_0}(\xi)|^p + |\widehat{f_1}(\xi)|^p \bigg)^{q/p} \,\textup{d}\xi \bigg)^{1/q} \\
& \leq \big( \|\widehat{f_0}\|_{\textup{L}^q(\T^{d-1})}^p + \|\widehat{f_1}\|_{\textup{L}^q(\T^{d-1})}^p \big)^{1/p}.
\end{align*}
Applying the induction hypothesis to $f_0$ and $f_1$,
\[ \|\widehat{f_0}\|_{\textup{L}^q(\T^{d-1})} \leq \|f_0\|_{\ell^p(\{0,1\}^{d-1})}, \quad
\|\widehat{f_1}\|_{\textup{L}^q(\T^{d-1})} \leq \|f_1\|_{\ell^p(\{0,1\}^{d-1})}, \]
we finally deduce 
\[ \|\widehat{f}\|_{\textup{L}^q(\T^d)} 
\leq \big( \|f_0\|_{\ell^p(\{0,1\}^{d-1})}^p + \|f_1\|_{\ell^p(\{0,1\}^{d-1})}^p \big)^{1/p}
= \|f\|_{\ell^p(\{0,1\}^d)} \]
and thus complete the induction. 


\section{Proof of Lemma \ref{lm:twopoint}}

Take some $q\in(2,\infty)$ and let $p=p(q)\in(1,2)$ be defined as \eqref{eq:pqequality}, so that \eqref{eq:improvedrange} is satisfied with an equality. 
It is sufficient to establish \eqref{eq:twopointineq} only for this value of $p$, since lowering $p$ only increases the right hand side.

Motivated by a computation from the introductory section, we substitute $a=\alpha^p$ and $b=\beta^p$, so that \eqref{eq:twopointineq} becomes
\[ \int_0^1 |a^{1/p} e^{2\pi i t} + b^{1/p}|^q \,\textup{d}t \leq (a + b)^{q/p}. \]
Both sides of the above inequality are homogeneous of degree $q/p$ in $(a,b)$, so we can simply just define $F_q\colon[0,1]\to(0,\infty)$ as
\begin{align}
F_q(x) & := \int_0^1 \big|x^{1/p} e^{2\pi i t} + (1-x)^{1/p}\big|^q \,\textup{d}t \nonumber \\
& \, = \int_{0}^{1} \big( (1-x)^{2/p} + x^{2/p} + 2 (1-x)^{1/p} x^{1/p} \cos 2\pi t \big)^{q/2} \,\textup{d}t \label{eq:fformula1}
\end{align}
and prove
\begin{equation}\label{eq:mainineq}
F_q(x) \leq 1 \quad \text{for every } x\in[0,1];
\end{equation}
see Figure \ref{fig:funct2D} for an illustration of the particular case $q=4$.

\begin{figure}
\includegraphics[width=0.55\linewidth]{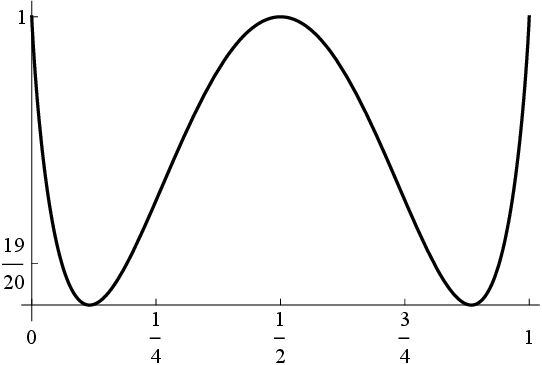}
\caption{Graph of $F_4$. The maximum $1$ is attained at $0$, $1/2$, and $1$.}
\label{fig:funct2D}
\end{figure}

The following lemma is a direct generalization of \cite[Lemma 5]{Kov23}, but a different proof is needed, since $F_q$ no longer has an expansion into a finite sum.

\begin{lemma}\label{lm:ode}
The function $F_q$ satisfies the second-order ordinary differential equation
\begin{equation}\label{eq:theODE}
a_q(x) F_q''(x) + b_q(x) F_q'(x) + c_q(x) F_q(x) = 0 \quad \text{for } x\in\Big(0,\frac{1}{2}\Big), 
\end{equation}
where $a_q,b_q,c_q\colon(0,1/2)\to\mathbb{R}$ are defined as
\begin{align*}
a_q(x) & := p^2 (1-x)^2 x^2 \big((1-x)^{2/p}-x^{2/p}\big), \\
b_q(x) & := p (1-x) x \Big((1-x)^{2/p}\big(p (1-2x) + 2 q x\big) + x^{2/p}\big(p (2x-1) + 2q(1-x)\big)\Big), \\
c_q(x) & := q (1-x)^{2/p} x \big(p(1-x) + q x\big) - q  x^{2/p}(1-x) \big(p x + q(1-x)\big).
\end{align*}
\end{lemma}

\begin{proof}
The Legendre function $P_\nu$ of degree $\nu$ has the integral representation
\begin{equation}\label{eq:fformula2}
P_{\nu}(z) = \int_{0}^{1} \big( z + \sqrt{z^2-1} \cos(2\pi t)\big)^{\nu} \,\textup{d} t 
\end{equation}
for $z\in(1,\infty)$; see \cite[Formula 14.12.7]{NIST}.
It also satisfies the differential equation
\begin{equation}\label{eq:LegendreODE}
(1-z^2) P_{\nu}''(z) - 2 z P_{\nu}'(z) + {\nu}({\nu}+1) P_{\nu}(z) = 0 
\end{equation}
for $z>1$; see \cite[Formula 14.2.1]{NIST}.

Denoting
\[ \varphi(x) := \frac{(1-x)^{2/p}+x^{2/p}}{(1-x)^{2/p}-x^{2/p}},  \]
we can compare \eqref{eq:fformula1} and \eqref{eq:fformula2} to derive
\begin{equation}\label{eq:fqexpressed}
F_q(x) = \big((1-x)^{2/p}-x^{2/p}\big)^{q/2} P_{q/2}(\varphi(x)).
\end{equation}
Then we can finish the proof using symbolic computation in Mathematica as follows. First, direct symbolic differentiation expresses
\begin{itemize}
\item $F'_q(x)$ in terms of $P_{q/2}(\varphi(x))$ and $P'_{q/2}(\varphi(x))$, 
\item $F''_q(x)$ in terms of $P_{q/2}(\varphi(x))$, $P'_{q/2}(\varphi(x))$, and $P''_{q/2}(\varphi(x))$.
\end{itemize}
Next, we use the Legendre ODE \eqref{eq:LegendreODE} to write
\begin{itemize}
\item $P''_{q/2}(\varphi(x))$ in terms of $P_{q/2}(\varphi(x))$ and $P'_{q/2}(\varphi(x))$,
\end{itemize}
and plug all these into the left hand side of \eqref{eq:theODE}.
Mathematica command \verb+Simplify+ perfoms algebraic manipulation sufficient to show that this is in fact identically $0$.

Alternatively, we can start with $F_q$ expressed as in \eqref{eq:fqexpressed} and immediately apply Mathematica command \verb+FullSimplify+ to the left hand side of \eqref{eq:theODE}. Once again we end up with $0$, but this fully automated verification does not reveal what properties of the Legendre functions were used during the simplification. It turns out that Mathematica first expresses their derivatives in terms of Legendre functions of varying degrees \cite[Formula 14.10.5]{NIST} and then relates the latter via the recurrent relation \cite[Formula 14.10.3]{NIST}.
\end{proof}


Before we begin the proof of \eqref{eq:mainineq}, let us observe the behavior of $F_q$ at $x=0,1/2$.

\begin{lemma}\label{lm:perturbative}
We have
\begin{align*}
F_q(0) = 1, & \qquad F_q(\varepsilon) = 1 - \frac{q}{p} \varepsilon + o(\varepsilon) \quad \text{as } \varepsilon\to0^+, \\
F_q\Big(\frac{1}{2}\Big) = 1, & \qquad F_q\Big(\frac{1}{2}-\varepsilon\Big) = 1 - \frac{2q^2}{p(q-1)} \Big(\underbrace{1-\frac{1}{p}-\frac{1}{q}}_{>0}\Big) \varepsilon^2 + o(\varepsilon^2) \quad \text{as } \varepsilon\to0^+.
\end{align*}
\end{lemma}

\begin{proof}
Substitution of $x=\varepsilon$ into the integrand of \eqref{eq:fformula1} gives
\begin{align*}
 \Big( 1 + 2 \varepsilon^{1/p} \cos 2\pi t - \frac{2}{p}\varepsilon + o(\varepsilon) \Big)^{q/2} 
 = 1 + q \varepsilon^{1/p} \cos 2\pi t - \frac{q}{p}\varepsilon + o(\varepsilon) 
\end{align*}
as $\varepsilon\to0^+$, uniformly in $t\in\R$, so integrating over $t\in[0,1]$ we deduce the first asymptotic formula.

Substituting $x=1/2-\varepsilon$ into the same integrand we get
\begin{align*}
& 2^{-q/p} \big( (1+2\varepsilon)^{2/p} + (1-2\varepsilon)^{2/p} + 2 (1-4\varepsilon^2)^{1/p} \cos 2\pi t \big)^{q/2} \\
& = 2^{q/2-q/p} \Big( 1 + \cos 2\pi t + \varepsilon^2 \frac{4}{p^2} (2-p-p\cos 2\pi t) + o(\varepsilon^2) \Big)^{q/2} \\
& = 2^{q/2-q/p} \Big( (1 + \cos 2\pi t)^{q/2} + \varepsilon^2 \frac{2q}{p^2} (1 + \cos 2\pi t)^{q/2-1} (2-p-p\cos 2\pi t) \Big) + o(\varepsilon^2)
\end{align*}
as $\varepsilon\to0^+$, uniformly in $t$ again.
Integrating this over $t\in[0,1]$ and using the same identities as in Section \ref{sec:proofthm}, we end up with
\[ 2^{-q/p} \binom{q}{q/2} - \varepsilon^2 \frac{2q^2}{p(q-1)} \Big(1-\frac{1}{p}-\frac{1}{q}\Big) 2^{-q/p} \binom{q}{q/2}. \]
Finally, the second asymptotic formula follows by recalling \eqref{eq:pqequality} and consequently replacing $\binom{q}{q/2}$ with $2^{q/p}$.

Both computations are trivially also valid for $\varepsilon=0$, revealing the values of $F_q$ at $0$ and $1/2$.
\end{proof}

Now we can finalize the proof of inequality \eqref{eq:mainineq}.
Notice that $F_q$ is symmetric about $1/2$, i.e., $F_q(1-x)=F_q(x)$ for $x\in[0,1]$.
Because of Lemma \ref{lm:perturbative} we know that $F_q$ has strict local maxima at $0$ and $1/2$, the one at $0$ being only one-sided.
Our goal is to show $F_q(x)\leq 1$ whenever $x\in(0,1/2)$. Suppose now, for contradiction, that there exists a point $x_{\max}\in (0,1/2)$ at which $F_q$ attains a local maximum.\footnote{One may further assume $F_q(x_{\max})>1$ in order to achieve the goal, but the argument in fact yields a stronger conclusion: $F_q$ cannot attain a local maximum in $(0,1/2)$ at all.}
Then $F_q$ would also attain local minima at some points $x_{\min,1}\in(0,x_{\max})$ and $x_{\min,2}\in(x_{\max},1/2)$; see Figure \ref{fig:minima}.
Since
\[ F'_q(x_{\min,1}) = F'_q(x_{\max}) = F'_q(x_{\min,2}) = 0 \]
and
\[ F''_q(x_{\min,1}) \geq 0, \quad F''_q(x_{\max}) \leq 0, \quad F''_q(x_{\min,2}) \geq 0, \]
the differential equation from Lemma \ref{lm:ode} and strict positivity of $F_q$ and $a_q$ on $(0,1/2)$ give
\[ c_q(x_{\min,1}) \leq 0, \quad c_q(x_{\max}) \geq 0, \quad c_q(x_{\min,2}) \leq 0. \]
In particular, by the intermediate value theorem, we see that $c_q$ has at least two zeros in the interval $(0,1/2)$.

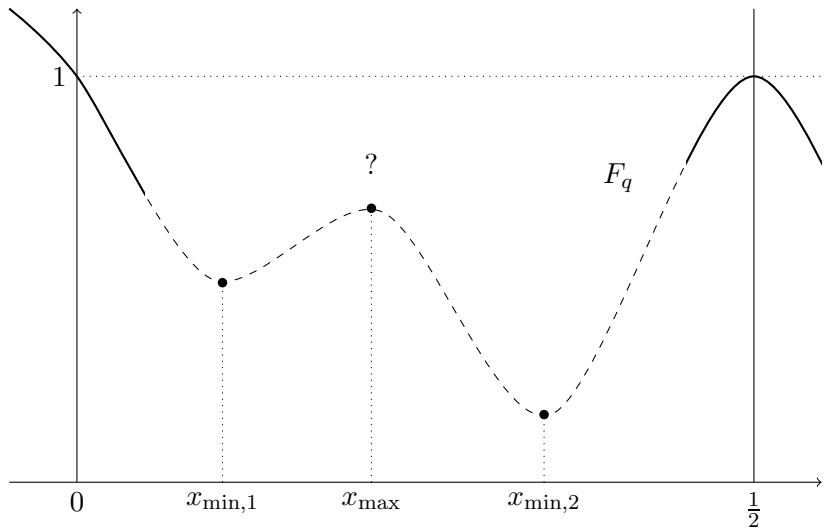
\begin{figure}
\centering
\begin{tikzpicture}[scale=9]


  \draw[->] (0,0.4) -- (0,1.1); 
  \draw (1,0.4)--(1,1.1);
  \draw[dotted](0,1)--(1.1,1);
  \draw[->] (-0.1,0.4)--(1.1,0.4);

\begin{scope}
  \clip (-0.2,0.4)--(1.1,0.4)--(1.1,1.2)--(-0.2,1.2);
  \draw[black, smooth, dashed] plot coordinates {
  (-0.1, 1.1) (0,1) (0.2,0.7) (0.45,0.8) (0.7,0.5) 
  (1,1)
  (1.3,0.5) (1.55,0.8) (1.8,0.7) (2,1) (2.1,1.1) 
  };
\end{scope}

\begin{scope}
  \clip (-0.2,0.4)--(0.1,0.4)--(0.1,1.2)--(-0.2,1.2);
  \draw[thick, black, smooth] plot coordinates {
  (-0.1, 1.1) (0,1) (0.2,0.7) (0.45,0.8) (0.7,0.5) 
  (1,1)
  (1.3,0.5) (1.55,0.8) (1.8,0.7) (2,1) (2.1,1.1) 
  };
\end{scope}

\begin{scope}
  \clip (0.9,0.4)--(1.1,0.4)--(1.1,1.2)--(0.9,1.2);
  \draw[thick, black, smooth] plot coordinates {
  (-0.1, 1.1) (0,1) (0.2,0.7) (0.45,0.8) (0.7,0.5) 
  (1,1)
  (1.3,0.5) (1.55,0.8) (1.8,0.7) (2,1) (2.1,1.1) 
  };
\end{scope}
\node [below] at (0,0.4) {$0$};
\node [left] at (0,1) {$1$};
\node [below] at (1,0.4) {$\frac12$};
\node at (0.8,0.85) {$\textcolor{black}{F_q}$};

\def\x{0.435}
\fill (\x,0.805) circle (0.2pt);  
\draw [dotted](\x,0.4)--(\x,0.805);
\node [below] at (\x,0.4) {$x_{\max}$};

\node at (\x,0.87) {$?$};

\def\y{0.215}
\fill (\y,0.695) circle (0.2pt);
\draw [dotted](\y,0.4)--(\y,0.69);
\node [below] at (\y,0.4) {$x_{\min,1}$};

\def\z{0.69}
\fill (\z,0.5) circle (0.2pt);
\draw [dotted](\z,0.4)--(\z,0.5);
\node [below] at (\z,0.4) {$x_{\min,2}$};

\end{tikzpicture}
\caption{A sketch of the function $F_q$. While the behavior near the points $x=0,1/2$ is well understood, the assumption that $F_q$ also attains its maximum at some point $x_{\max} \in (0,1/2)$ leads to a contradiction.}
\label{fig:minima}
\end{figure}

Recall that the point $(1/p,1/q)$ is at the lower boundary of the range \eqref{eq:improvedrange}, so the two Lebesgue exponents satisfy the sub-H\"{o}lder condition \eqref{eq:subduality}.
Let us substitute $x=1/(y+1)$ for $y\in(1,\infty)$; this change of variables was used in \cite{KT17} and \cite{DGIM21}.
An easy computation gives
\[ c_{q}\Big(\frac{1}{y+1}\Big) = q y (y+1)^{-2/p-2} \big( p y^{2/p} + q y^{2/p-1} -q y - p \big), \]
so, from the previous discussion, we know that $\phi\colon(0,\infty)\to\R$ defined as
\[ \phi(y) := p y^{2/p} + q y^{2/p-1} -q y - p \]
has at least two zeros in $(1,\infty)$.
However, by direct computation, we obtain
\[ \phi'(y) = 2y^{2/p-1} + q \Big(\frac{2}{p}-1\Big) y^{2/p-2} - q \]
and
\[ \phi''(y) = \frac{2(2-p)q}{p} y^{2/p-3} \Big(\frac{y}{q}+\frac{1}{p}-1\Big). \]
Denoting
\[ y_0 := q \Big(1-\frac{1}{p}\Big) \stackrel{\eqref{eq:subduality}}{>} 1, \]
we see that $\phi'$ is strictly decreasing on $(0,y_0]$ and strictly increasing on $[y_0,\infty)$.
Since
\[ \phi'(1) = -2q \Big(1-\frac{1}{p}-\frac{1}{q}\Big) \stackrel{\eqref{eq:subduality}}{<} 0 
\quad\text{and}\quad \lim_{y\to\infty} \phi'(y)=\infty, \]
we see that there exists $y_1\in(y_0,\infty)$ such that $\phi$ is strictly decreasing on $[1,y_1]$ and strictly increasing on $[y_1,\infty)$.
Combining this with $\phi(1)=0$, we conclude that $\phi$ has a unique zero in $(1,\infty)$, which contradicts the previous claim about the existence of two zeros.
This contradiction finally establishes \eqref{eq:mainineq}.


\section{Proof of Corollary \ref{cor:additive}}

The estimate \eqref{eq:additiver} is just a special case of \eqref{eq:improvedHY},
\[ E_{\kappa}(A) \leq \|\mathbbm{1}_A\|_{\ell^p(\Z^d)}^{2\kappa} = |A|^{q/p} = |A|^r, \]
since $r=q/p$, when $q=2\kappa$ and $p$ is chosen as in \eqref{eq:pqequality}.

To see that $A=\{0,1\}^d$ attains the equality, we recall the computation from the beginning of Section \ref{sec:proofthm}, which implies that \eqref{eq:auxnecess} holds with an equality for $p$ given by \eqref{eq:pqequality}.
Now we have, by Fubini's theorem,
\[ E_{\kappa}(\{0,1\}^d) = \Big( \int_0^1 |1+e^{-2\pi i t}|^{2\kappa} \,\textup{d}t \Big)^d = 2^{dq/p} = (2^d)^r = \big|\{0,1\}^d\big|^r. \]


\section{Proof of Corollary \ref{cor:restriction}}

\subsection{Necessity of the condition on \texorpdfstring{$p,q$}{p,q}}
Suppose that \eqref{eq:improvedHY3} holds with a constant $C_{p,q}$ independent of the dimension.
Take $r\in\N$ and $g\in\ell^{q'}(\T^d)$ and apply \eqref{eq:improvedHY3} to the $rd$-dimensional function
\[ G(\xi_1,\xi_2,\ldots,\xi_r) := g(\xi_1) g(\xi_2) \cdots g(\xi_r), \quad \xi_1,\xi_2,\ldots,\xi_r \in \T^d, \]
to obtain
\[ \big\|\widehat{g}\big|_{\{0,1\}^d}\big\|_{\ell^{p'}(\{0,1\}^d)}^r \leq C_{p,q} \|g\|_{\textup{L}^{q'}(\T^d)}^r. \]
Taking the $r$-th roots of both sides and letting $r\to\infty$ we recover \eqref{eq:improvedHY2}, i.e., the constant becomes $1$.
Now the implication \eqref{thmit2}$\implies$\eqref{thmit3} from Theorem \ref{thm:main} confirms that the exponents $p,q$ need to belong to the range \eqref{eq:improvedrange}.

\subsection{Sufficiency of the condition on \texorpdfstring{$p,q$}{p,q}}
This follows from the implication \eqref{thmit3}$\implies$\eqref{thmit2} from Theorem \ref{thm:main} and we immediately also obtain the sharp constant $C_{p,q}=1$.


\section{Proof of Corollary \ref{cor:entropy}}
\label{sec:entropy}

In this section we again assume that $q\in(2,\infty)$ and that $p=p(q)$ is always defined by \eqref{eq:pqequality}.
If we write
\[ q = 2 + \varepsilon, \]
then, using Taylor's expansion of $z\mapsto \log_2 \Gamma(z)$ around $z=2$ and $z=3$ in terms of the digamma function 
\[ \psi(z) := \frac{\textup{d}}{\textup{d}z} \ln \Gamma(z), \]
we obtain
\[ p = \frac{2+\varepsilon}{1+(\psi(3)/\ln 2)\varepsilon-2(\psi(2)/\ln 2)(\varepsilon/2)+o(\varepsilon)} \quad\text{as } \varepsilon\to0^+. \]
Since \cite[Formula 5.5.2]{NIST} specializes to	
\[ \psi(3) - \psi(2) = \frac{1}{2}, \]
we get
\begin{align}
p & = \frac{2 + \varepsilon}{1 + \varepsilon/(2\ln 2) + o(\varepsilon)} \nonumber \\
& = 2 - \Big(\frac{1}{\ln 2} - 1\Big) \varepsilon + o(\varepsilon) \quad\text{as } \varepsilon\to0^+. \label{eq:asymofp}
\end{align}

Take a function $f\colon\Z^d\to\mathbb{C}$ supported in $\{0,1\}^d$ and normalized as in \eqref{eq:normalization}. Denote
\[ S := \{ x\in\Z^d : f(x)\neq 0 \}, \quad T := \{ \xi\in\T^d : \widehat{f}(\xi)\neq 0 \}. \]
Applying \eqref{eq:improvedHY} from Theorem \ref{thm:main} with $p,q$ as above gives
\begin{equation}\label{eq:applofthm}
\Big( \int_T |\widehat{f}(\xi)|^q \,\textup{d}\xi \Big)^{1/q} \leq \Big( \sum_{x\in S} |f(x)|^p \Big)^{1/p}.
\end{equation}
Note that $|\widehat{f}|$ is bounded, so
\[ |\widehat{f}(\xi)|^q = |\widehat{f}(\xi)|^2 + \varepsilon |\widehat{f}(\xi)|^2 \ln |\widehat{f}(\xi)| + o(\varepsilon) \]
as $\varepsilon\to0^+$, uniformly in $\xi\in T$, and then
\[ \Big( \int_T |\widehat{f}(\xi)|^q \,\textup{d}\xi \Big)^{1/q} = 1 + \frac{\varepsilon}{2} \int_T |\widehat{f}(\xi)|^2 \ln |\widehat{f}(\xi)| \,\textup{d}\xi + o(\varepsilon) \quad\text{as }\varepsilon\to0^+, \]
thanks to our normalization
\[ \int_T |\widehat{f}(\xi)|^2 \,\textup{d}\xi = 1. \]
Also,
\[ |f(x)|^p \stackrel{\eqref{eq:asymofp}}{=} |f(x)|^2 - \Big(\frac{1}{\ln 2} - 1\Big) \varepsilon |f(x)|^2 \ln |f(x)| + o(\varepsilon) \quad\text{as }\varepsilon\to0^+, \]
which gives
\[ \Big( \sum_{x\in S} |f(x)|^p \Big)^{1/p} \stackrel{\eqref{eq:normalization}}{=} 1 - \Big(\frac{1}{\ln 2} - 1\Big) \frac{\varepsilon}{2}\sum_{x\in S} |f(x)|^2 \ln |f(x)| + o(\varepsilon) \quad\text{as }\varepsilon\to0^+. \]
Recall inequality \eqref{eq:applofthm}, subtract $1$, divide by $\varepsilon/(4\log_2 e)$, and let $\varepsilon\to0^+$ to obtain
\[ -\textup{H}_{\T^d}(|\widehat{f}|^2) \leq \Big(\frac{1}{\ln 2} - 1\Big) \,\textup{H}_{\Z^d}(|f|^2). \]

The equality will be attained for
\[ f = 2^{-d/2} \mathbbm{1}_{\{0,1\}^d}, \]
when
\[ \textup{H}_{\Z^d}(|f|^2) = - \sum_{x\in\{0,1\}^d} 2^{-d} \log_2 2^{-d} = d \neq 0. \]
Namely,
\[ \widehat{f}(\xi) = 2^{-d/2} \prod_{j=1}^{d} \big(1 + e^{-2\pi i \xi_j}\big) \]
for $\xi=(\xi_1,\xi_2,\ldots,\xi_d)\in\T^d$. Then Fubini's theorem on $\T^d$ and integration by parts enable us to compute
\begin{align*}
-\textup{H}_{\T^d}(|\widehat{f}|^2) & = \int_{\T^d} 2^{d} \Big(\prod_{j=1}^{d} \cos^2 \pi\xi_j\Big) \Big(d + \prod_{j=1}^{d} \log_2 \cos^2 \pi\xi_j\Big) \,\textup{d}\xi \\
& = d + 2d \int_0^1 (\cos^2 \pi t) (\log_2 \cos^2 \pi t) \,\textup{d}t \\
& = d + 2d\Big(\frac{1}{2\ln 2} - 1\Big) = d \Big(\frac{1}{\ln 2} - 1\Big).
\end{align*}
Therefore,
\[ \textup{H}_{\T^d}(|\widehat{f}|^2) + \Big(\frac{1}{\ln 2}-1\Big) \textup{H}_{\Z^d}(|f|^2) = 0, \]
as desired.


\section{Proof of Theorem~\ref{thm:Young} assuming Lemma~\ref{lm:twopointYoung}}

\subsection{Necessity of the condition on \texorpdfstring{$p,q$}{p,q}}
Suppose that $1\leq p,q<\infty$.
We know that \eqref{eq:improvedYoung} holds for one-dimensional functions $f,g\colon\Z\to\mathbb{C}$ such that $f(0)=f(1)=g(0)=g(1)=1$, which vanish outside of $\{0,1\}$.
Since
\[ (f\ast g)(x) = 
\begin{cases}
2 & \text{if } x=1, \\
1 & \text{if } x=0 \text{ or } x=2, \\
0 & \text{otherwise},
\end{cases} \]
the inequality \eqref{eq:improvedYoung} actually gives
\[ (1^q + 2^q + 1^q)^{1/q} \leq 2^{1/p}2^{1/p}, \]
which transforms into \eqref{eq:rangeYoung2}.

\subsection{Sufficiency of the condition on \texorpdfstring{$p,q$}{p,q}}
Take $p\in(1,2)$ and $q\in(1,\infty)$ that satisfy \eqref{eq:rangeYoung2}, as the cases $q=1,\infty$ are trivial.
Replacing $f$ and $g$ with $|f|$ and $|g|$ we see that it is sufficient to work with nonnegative functions only.
Once again we proceed by the mathematical induction on the dimension $d$.
When $d=1$ estimate \eqref{eq:improvedYoung} is just \eqref{eq:twopointineqYoung} with
\[ \alpha_0 = f(0), \ \alpha_1 = f(1), \ \beta_0 = g(0), \ \beta_1 = g(1), \]
which holds in the range \eqref{eq:rangeYoung2} by Lemma \ref{lm:twopointYoung}.

In the induction step we take $d\geq 2$ and functions $f,g\colon\Z^d\to[0,\infty)$ supported in $\{0,1\}^d$.
Define $f_0,f_1,g_0,g_1,\widetilde{f},\widetilde{g}\colon\Z^{d-1}\to\mathbb{C}$ by
\begin{align*}
& f_0(x):=f(x,0), \quad f_1(x):=f(x,1), \quad g_0(x):=g(x,0), \quad g_1(x):=g(x,1), \\
& \widetilde{f}(x):=\big(f_0(x)^p + f_1(x)^p\big)^{1/p}, \quad \widetilde{g}(x):=\big(g_0(x)^p + g_1(x)^p\big)^{1/p} 
\end{align*}
for every $x\in\Z^{d-1}$.
Then $f\ast g$ is supported on $\{0,1,2\}^d$ and for each $x\in\Z^{d-1}$ we have
\begin{align*}
(f\ast g)(x,0) & = \sum_{y\in\Z^{d-1}} f_0(x-y) g_0(y), \\
(f\ast g)(x,1) & = \sum_{y\in\Z^{d-1}} \big( f_0(x-y) g_1(y) + f_1(x-y) g_0(y) \big), \\
(f\ast g)(x,2) & = \sum_{y\in\Z^{d-1}} f_1(x-y) g_1(y).
\end{align*}
The $\ell^q$-norm of $f\ast g$ will be estimated ``slice-wise,'' so first observe that Minkowski's inequality followed by \eqref{eq:twopointineqYoung} with
\[ \alpha_0 = f_0(x-y), \ \alpha_1 = f_1(x-y), \ \beta_0 = g_0(y), \ \beta_1 = g_1(y) \]
gives
\begin{align*}
& \big( (f\ast g)(x,0)^q + (f\ast g)(x,1)^q + (f\ast g)(x,2)^q \big)^{1/q} \\
& \leq \sum_{y\in\Z^{d-1}} \Big( \big( f_0(x-y) g_0(y) \big)^q + \big( f_0(x-y) g_1(y) + f_1(x-y) g_0(y) \big)^q + \big( f_1(x-y) g_1(y) \big)^q \Big)^{1/q} \\
& \stackrel{\eqref{eq:twopointineqYoung}}{\leq} \sum_{y\in\Z^{d-1}} \widetilde{f}(x-y) \widetilde{g}(y) = \big(\widetilde{f} \ast \widetilde{g}\big)(x).
\end{align*}
Finally, we take the $\ell^q$-norm in $x$ and use the induction hypothesis applied to $\widetilde{f}$ and $\widetilde{g}$:
\[ \|f\ast g\|_{\ell^q(\Z^d)} 
\leq \big\| \widetilde{f} \ast \widetilde{g} \big\|_{\ell^q(\Z^{d-1})} 
\leq \big\| \widetilde{f} \big\|_{\ell^p(\Z^{d-1})} \big\| \widetilde{g} \big\|_{\ell^p(\Z^{d-1})}
= \|f\|_{\ell^p(\Z^d)} \|g\|_{\ell^p(\Z^d)}. \]
This completes the induction step.


\section{Proof of Lemma~\ref{lm:twopointYoung}}

For any $q\in(1,\infty)$ we consider the endpoint exponent $p=p(q)$ only, given by
\begin{equation}\label{eq:pqeqYoung}
p = \frac{2q}{\log_2(2^q+2)},
\end{equation}
i.e., $(1/p,1/q)$ is from the upper boundary of the region \eqref{eq:rangeYoung2} depicted in the right half of Figure \ref{fig:Young}. Note
\begin{equation}\label{eq:whatarepandq}
1<p<\min\{q,2\}. 
\end{equation}
Substituting
\[ a_j=\alpha_j^p, \quad b_j=\beta_j^p \]
for $j=0,1$, inequality \eqref{eq:twopointineqYoung} turns into
\[ a_0^{q/p} b_0^{q/p} + \big( a_0^{1/p} b_1^{1/p} + a_1^{1/p} b_0^{1/p} \big)^{q} + a_1^{q/p} b_1^{q/p}
\leq (a_0 + a_1)^{q/p} (b_0 + b_1)^{q/p}. \]
Since both sides are homogeneous functions of order $q/p$ in both $(a_0,a_1)$ and $(b_0,b_1)$, it is enough to show that  $G_q\colon[0,1]^2\to(0,\infty)$ defined as
\[ G_q(x,y) := (1-x)^{q/p} (1-y)^{q/p} + \big( (1-x)^{1/p} y^{1/p} + x^{1/p} (1-y)^{1/p} \big)^{q} + x^{q/p} y^{q/p} \]
satisfies
\begin{equation}\label{eq:mainineqYoung}
G_q(x,y) \leq 1 \quad \text{for every } (x,y)\in[0,1]^2;
\end{equation}
see Figure \ref{fig:funct3D} for an illustration of the particular case $q=2$.
Obviously, $G_q(0,0)=G_q(1,0)=G_q(0,1)=G_q(1,1)=1$, while the choice \eqref{eq:pqeqYoung} guarantees $G_q(1/2,1/2)=1$.
Also observe that
\[ G_q(x,y) = G_q(y,x) = G_q(1-y,1-x) \]
for $(x,y)\in[0,1]^2$, i.e., $G_q$ is symmetric about the lines $y=x$ and $y=1-x$. 

\begin{figure}
\includegraphics[width=0.55\linewidth]{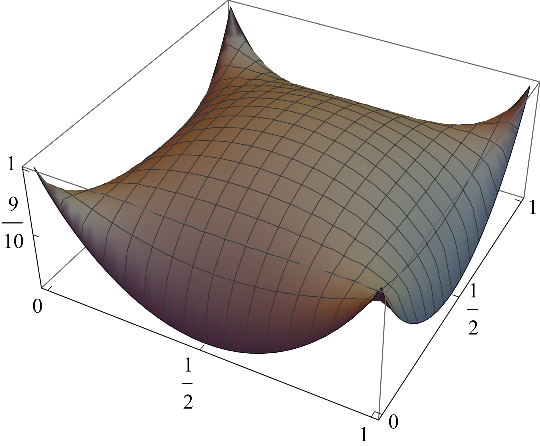}
\caption{Graph of $G_2$. The maximum $1$ is attained at $(0,0)$, $(0,1)$, $(1,0)$, $(1,1)$, and $(1/2,1/2)$.}
\label{fig:funct3D}
\end{figure}

Before we begin the proof of \eqref{eq:mainineqYoung} we need to establish two elementary inequalities.

\begin{lemma}\label{lm:cosh}
For exponents $p,q$ as before and for every $t\in\R$ we have
\begin{equation}\label{eq:thetwoineq}
2^{2q/p} \bigg( \Big(\cosh\frac{pt}{2q}\Big)^{2q/p} - 1 \bigg)
\geq 2^q \bigg( \Big(\cosh\frac{t}{q}\Big)^{q} - 1 \bigg)
\end{equation}
and
\begin{equation}\label{eq:thetwoineq2}
\Big(2\cosh\frac{t}{q}\Big)^{q} > \frac{2(2q-p)}{p} \cosh\frac{pt}{q}.
\end{equation}
\end{lemma}

\begin{proof}
Note that \eqref{eq:thetwoineq} becomes an equality for $t=0$, while the hyperbolic cosine is an even function, which enables us to restrict our consideration to $t\in(0,\infty)$.
For the proof of \eqref{eq:thetwoineq} define
\[ \Phi\colon(0,\infty)^2 \to \R, \quad
\Phi(s,t) := 2^s \bigg(\Big(\cosh\frac{t}{s}\Big)^s-1\bigg), \]
so that the estimate can be written simply as
\begin{equation}\label{eq:whatweneed}
\Phi\Big(\frac{2q}{p},t\Big) \geq \Phi(q,t) \quad \text{for } t\in(0,\infty).
\end{equation}
For $s>2$ and $t>0$ differentiation gives 
\[ 2^{-s} \partial_s\Phi(s,t)
= \Big(\cosh\frac{t}{s}\Big)^s \bigg( \ln\Big(2\cosh\frac{t}{s}\Big) - \frac{t}{s} \tanh\frac{t}{s} \bigg) - \ln 2
\geq \varphi\Big(\frac{t}{s}\Big), \]
where $\varphi\colon[0,\infty)\to\R$ is now a concrete one-dimensional function, 
\[ \varphi(u) := (\cosh u)^2 \big( \ln(2\cosh u) - u \tanh u \big) - \ln 2. \]
It is easily seen to be positive on $(0,\infty)$, because
\begin{align*}
& \varphi(0)=0, \quad 
\varphi'(u) = \sinh(2u) \gamma(u), \quad 
\gamma(u) := \ln(2\cosh u)-u\coth(2u), \\ 
& \gamma'(u) = \frac{2u-\sinh(2u)}{(\sinh(2u))^2}<0, \quad 
\lim_{u\to 0^+}\gamma(u)=\ln 2-\frac{1}{2} > 0, \quad 
\lim_{u\to\infty}\gamma(u) =0
\end{align*}
imply $\gamma(u)>0$, $\varphi'(u)>0$, and then also $\varphi(u)>0$ for every $u\in(0,\infty)$.
Thus, $s\mapsto\Phi(s,t)$ is increasing on $[2,\infty)$ for any fixed $t>0$.
Since 
\[ \frac{2q}{p} \stackrel{\eqref{eq:whatarepandq}}{>} \max\{q,2\}, \]
we are done proving \eqref{eq:whatweneed} in the case $q\geq 2$, while in the case $1<q<2$ we still need to show
\[ \Phi(2,t) \geq \Phi(q,t) \quad \text{for } t\in(0,\infty). \]
However,
\[ \Phi(2,t) = 4 \bigg(\Big(\cosh\frac{t}{2}\Big)^2-1\bigg) = 2 (\cosh t - 1) = \Phi(1,t) \]
and it is again natural to substitute $u=t/q$, so it remains to see that 
\[ \eta(q,u) := \Phi(1,qu) - \Phi(q,qu) 
= 2 (\cosh(qu) - 1) - 2^q \big( (\cosh u)^q-1 \big) \]
is always nonnegative when $1<q<2$ and $u>0$. Here we have $\eta(q,0)=0$,
\[ \partial_u \eta(q,u) = 2q\sinh(qu) - 2^q q (\cosh u)^{q-1} \sinh u, \]
and it is sufficient to verify $\partial_u \eta(q,u)\geq 0$ for $u>0$, but this is equivalent to
\begin{equation}\label{eq:thelastphi}
\phi(q,u) := \ln\sinh(qu) - (q-1)\ln(2\cosh u) - \ln\sinh u \geq 0. 
\end{equation}
Finally,
\[ \partial_q^2 \phi(q,u) = - \Big(\frac{u}{\sinh(qu)}\Big)^2 \leq 0, \]
so $q\mapsto\phi(q,u)$ is concave on $[1,2]$ for a fixed $u>0$ and then
\[ \phi(1,u) = 0, \quad \phi(2,u) = \ln\frac{\sinh(2u)}{2\cosh u \sinh u} = 0 \]
imply that \eqref{eq:thelastphi} holds for every $q\in(1,2)$ and $u\in(0,\infty)$, completing the proof of \eqref{eq:thetwoineq}.

For the proof of \eqref{eq:thetwoineq2}, we fix $p,q\in(1,\infty)$ related by \eqref{eq:pqeqYoung}, recalling the right image in Figure \ref{fig:Young}. Note that the assignment $q\mapsto p(q)$ actually makes sense for all $q\in(0,\infty)$ and we can compute
\begin{equation}\label{eq:expprimeeq}
p'(q) = \frac{p}{q} \Big( 1 - \frac{p 2^q}{2(2^q + 2)} \Big), \quad p'(1)=\frac{3}{4}.
\end{equation}
Substitute $u=t/q$, transforming \eqref{eq:thetwoineq2} into
\begin{equation}\label{eq:thetwoineq2b}
\big(2\cosh u\big)^q > \frac{2(2q-p)}{p} \cosh(p u).
\end{equation}
Taking logarithms and defining $\widetilde{\Phi}\colon(0,\infty)\times[0,\infty)\to\R$ as
\[ \widetilde{\Phi}(q,u) := q\ln\cosh u - \ln\cosh(pu) + \ln\frac{p 2^{q-1}}{2q-p}, \]
the desired inequality \eqref{eq:thetwoineq2b} becomes 
$\widetilde{\Phi}(q,u)>0$ for $q>1$ and $u\geq0$.

Let us begin with the case of exponents satisfying $p^2\leq q$.
Observe
\begin{align*} 
\partial_u\widetilde{\Phi}(q,u) & = q \tanh u - p \tanh(pu) \\
& = p (\tanh u) \Big(\frac{q}{p} - \frac{\tanh(pu)}{\tanh u}\Big),
\end{align*}
that $u\mapsto \tanh(pu) / \tanh u$ is strictly decreasing, which follows from
\[ \frac{\textup{d}}{\textup{d}u} \ln\frac{\tanh(pu)}{\tanh u}
= \frac{2p}{\sinh(2pu)} - \frac{2}{\sinh(2u)} < 0, \]
and that
\[ \lim_{u\to0^+}\frac{\tanh(pu)}{\tanh u} = p, \quad \lim_{u\to\infty}\frac{\tanh(pu)}{\tanh u} = 1. \]
Since $q/p\geq p$, we see that $u\mapsto\widetilde{\Phi}(q,u)$ is strictly increasing on $[0,\infty)$, so
\[ \widetilde{\Phi}(q,0) = \ln \frac{2^{q-1}p}{2q-p} \stackrel{\eqref{eq:pqeqYoung}}{=} \ln \frac{2^{q-1}}{\log_2(2^{q-1}+1)} > 0 \]
implies that $\widetilde{\Phi}(q,u)$ is positive for $u>0$ as well, and \eqref{eq:thetwoineq2b} is confirmed.

Now we turn to the case $p^2>q$, which easily implies $q<4$ and this is all that we need in what follows.
Using \eqref{eq:expprimeeq} and differentiating once again we easily estimate\footnote{Reliable optimization of one-dimensional functions (with no other variables or parameters) is never a difficulty for Mathematica \cite{Mathematica}. Also see the very last part of the proof for an example of how (even higher-dimensional) concrete inequalities can be verified $100\%$ rigorously.} 
\begin{equation}\label{eq:expprimeineq}
0 < p'(q) \leq \frac{3p}{4q}, \quad \frac{2p}{5q^2} < -p''(q) < \frac{4p}{5q^2}
\end{equation}
for $1\leq q\leq4$.
Positivity of $\widetilde{\Phi}(q,u)$ for $q>1$ will follow from $\widetilde{\Phi}(1,u)=0$ once we can prove
\begin{equation}\label{eq:weneedPhi}
\partial_q \widetilde{\Phi}(q,u) > 0
\end{equation}
for $q\in(1,4)$ and $u\in[0,\infty)$.
Note that
\[ \partial_q \widetilde{\Phi}(q,u) = \widetilde{\eta}(q,u) + \frac{2}{2q-p}  \Big(p'(q)\frac{q}{p} - 1 \Big) + \ln 2, \]
where
\[ \widetilde{\eta}(q,u) := \ln\cosh u - p'(q) u \tanh (pu). \]
Furthermore,
\[ \partial_q \widetilde{\eta}(q,u) = - p''(q) u \tanh (pu) - p'(q)^2 \frac{u^2}{(\cosh (pu))^2} \stackrel{\eqref{eq:expprimeineq}}{\geq} \frac{1}{q^2}\widetilde{\varphi}(pu), \]
where $\widetilde{\varphi}$ is a concrete one-dimensional function given by
\[ \widetilde{\varphi}(s) := \frac{2}{5}s\tanh s - \frac{s^2}{(\cosh s)^2}. \]
It is easily seen that $\widetilde{\varphi}(s) \geq 0$ for $s\geq3$, which then implies $\partial_q \widetilde{\eta}(q,u)\geq0$ for $1<q<4$ and $u\geq 3$, and thus also
\begin{align*} 
\widetilde{\eta}(q,u) & \geq \widetilde{\eta}(1,u)  
\stackrel{\eqref{eq:expprimeeq}}{=} \ln\cosh u - \frac{3}{4} u \tanh u > 0
\end{align*}
for $q\in(1,4)$ and $u\in[3,\infty)$, by an easy one-dimensional optimization in $u$.
Finally,
\[ \partial_q \widetilde{\Phi}(q,u) \geq - \frac{p 2^q}{(2q-p)(2^q+2)} + \ln 2 \]
and yet another one-dimensional optimization, this time in $1\leq q\leq 4$, shows that this is always strictly positive, confirming \eqref{eq:weneedPhi} for $u\geq3$.

In the last part of the proof we verify \eqref{eq:weneedPhi} when $q\in[1,4]$ and $u\in[0,3]$.
Mathematica can perform numerical optimization of $\partial_q \widetilde{\Phi}(q,u)$ over $(q,u)\in[1,4]\times[0,3]$, computing its minimal value as
\[ 0.0293596409\ldots > \frac{1}{50}, \]
but this might not be completely rigorous. Thus, we rather reduce the desired strict inequality to finitely many single-number inequalities corresponding to nodes of a subdivision of that square.
We compute
\begin{align*}
\partial_q^2 \widetilde{\Phi}(q,u) & = - p''(q) u \tanh(pu) - p'(q)^2 \frac{u^2}{(\cosh(pu))^2} + p''(q) \frac{2q}{p(2q-p)} \\
& \quad - p'(q)^2 \frac{4q(q-p)}{p^2(2q-p)^2} + \frac{4}{(2q-p)^2}\big(1-p'(q)\big), \\
\partial_u \partial_q \widetilde{\Phi}(q,u) & = -p'(q) \tanh(pu) - p'(q) \frac{pu}{(\cosh(pu))^2} + \tanh u.
\end{align*}
Using \eqref{eq:expprimeineq} we easily get
\[ |\partial_q^2 \widetilde{\Phi}(q,u)| \leq 7,
\quad |\partial_u \partial_q \widetilde{\Phi}(q,u)| \leq 3 \]
for $1\leq q\leq 4$ and $0\leq u\leq 3$.
Thus, $(q,u)\mapsto\partial_q \widetilde{\Phi}(q,u)$ is a Lipschitz function on the square $[1,4]\times[0,3]$ with explicit horizontal and vertical Lipschitz constants, so it is sufficient to verify $\partial_q \widetilde{\Phi}(q,u)> 1/50$ on the two-dimensional subdivision
\[ \Big\{ (q,u) = \Big(1+\frac{i}{700},\frac{j}{300}\Big) \,:\, i\in\{0,1,2,\ldots,2100\},\ j\in\{0,1,2,\ldots,900\} \Big\}. \]
This is easily done in Mathematica, using the exact expressions and not just numerical approximations. It completes the proof of \eqref{eq:weneedPhi} and thus also that of \eqref{eq:thetwoineq2b}.
\end{proof}

\begin{lemma}\label{lm:boundary antidiagonal}
It holds that $G_q(x,y)\leq 1$ on the boundary of $[0,1]^2$ and on the anti-diagonal line $y=1-x$. 
\end{lemma}
\begin{proof}
Regarding the boundary, due to the symmetries of $G_q$ it is sufficient to consider the left side of the square. Substituting $x=0$ into $G_q(x,y)$ and using $q/p>1$ we obtain
\[ G_q(0,y) = (1-y)^{q/p} + y^{q/p} \leq (1-y) + y = 1. \]

Regarding the anti-diagonal of $[0,1]^2$, the goal is to establish
\[ G_q(x,1-x) = \big((1-x)^{2/p}+x^{2/p}\big)^q+2(1-x)^{q/p}x^{q/p}\leq 1 \]
for $x\in(0,1)$. 
Substituting 
\[ x = \frac{1}{1 + e^{pt/q}}, \quad t\in\R, \]
this estimate transforms into
\[ \big( e^{pt/q} + 1 \big)^{2q/p} \geq \big( e^{2t/q} + 1 \big)^{q} + 2 e^t \]
and, dividing by $e^t$, into
\[ \big( e^{pt/(2q)} + e^{-pt/(2q)} \big)^{2q/p} \geq \big( e^{t/q} + e^{-t/q} \big)^q + 2, \]
i.e.,
\[ \Big( 2 \cosh \frac{pt}{2q} \Big)^{2q/p} \geq \Big( 2 \cosh \frac{t}{q} \Big)^q + 2. \]
Finally note that \eqref{eq:pqeqYoung} can be rewritten as
$2^{2q/p} - 2^q = 2$,
so the last inequality is precisely the first of the two inequalities from Lemma \ref{lm:cosh}, namely \eqref{eq:thetwoineq}.
\end{proof}

In the rest of the proof, let us consider points $(x,y)$ in the open lower-right triangle with anti-diagonal removed,
\[ \Delta:=\big\{(x,y)\in (0,1)^2: y\leq x,\, y\not = 1-x\big\}, \]
shaded in Figure \ref{fig:two_max}.
Suppose now, for contradiction, that there exists a point $(x_{\max},y_{\max})\in \Delta$ at which $G_q$ attains a local maximum.\footnote{Again, one may further assume $G_q(x_{\max},y_{\max})>1$ in order to achieve the goal, but the argument in fact yields a stronger conclusion: $G_q$ cannot attain a local maximum in $\Delta$ at all.}
We analyze $G_q$ along curves 
\begin{equation}\label{e:dissecting curve}
\Theta_a := \Big\{ (x,y)\in(0,1)^2 \,:\, \frac{y}{1-y}=a\frac{x}{1-x} \Big\}
\end{equation}
for $0<a\leq 1$.
These curves are easily seen to partition the triangle $0<y\leq x<1$.
Remark \ref{rem:thePDE} will motivate this unusual choice of curves, but they also reflect the structure of $G_q$.
Simplifying the formula, we can write the equation of $\Theta_a$ as
\[ y=\theta_a(x):=\frac{ax}{1-x+ax}. \]

Choose $a\in (0,1]$ so that the curve passes through the point $(x_{\max},y_{\max})\in \Delta$. By the symmetry of both $G_q$ and $\Theta_a$, the reflected point $(x_{\max}',y_{\max}')\in \Delta$, obtained by mirroring across the line $y=1-x$, must also be a local maximum of $G_q$. Therefore, the one-dimensional function $G_q\vert_{\Theta_a}$ defined along the curve also has at least two local maxima in $\Theta_a\cap\Delta$. 
Moreover, by direct computation in Mathematica we see that
\[ \lim_{x\to0^+} \frac{\textup{d}}{\textup{d} x}G_q(x,\theta_a(x)) = -\frac{q}{p}(1+a) < 0. \]
Also note that $y=\theta_a(x)$ intersects $y=1-x$ at the point with $x$-coordinate $1/(1+\sqrt{a})$. 
We want to verify
\begin{equation}\label{eq:centralmaxofGq}
\frac{\textup{d}^2}{\textup{d} x^2}\Big\vert_{x=1/(1+\sqrt{a})} G_q(x,\theta_a(x)) < 0
\end{equation}
and Mathematica can easily compute this second derivative as
\[ \frac{2q}{p} a^{(q-p)/(2p)} \big(1+\sqrt{a}\big)^{2(p-q)/p}
\Big( -\big( a^{-1/(2p)} + a^{1/(2p)} \big)^q + \frac{q}{p} \big( a^{-1/2} + a^{1/2} \big) + \frac{2(q-p)}{p} \Big). \]
Then we perform the substitution
\[ a = e^{-2pt/q}, \quad t\in[0,\infty), \]
to see that \eqref{eq:centralmaxofGq} is equivalent with
\[ - \Big(2\cosh\frac{t}{q}\Big)^q + \frac{2q}{p} \cosh\frac{pt}{q} + \frac{2(q-p)}{p} < 0. \]
However, this is an immediate consequence of \eqref{eq:thetwoineq2} from Lemma \ref{lm:cosh} because $\cosh(pt/q)\geq1$.

All this and the symmetry of $G_q$ about $y=1-x$ imply that $G_q\vert_{\Theta_a}$ is decreasing\footnote{Monotonicity is considered with respect to the natural orientation of $\Theta_a$, from $(0,0)$ to $(1,1)$.} near $(0,0)$, increasing near $(1,1)$, and has yet another local maximum at the intersection of $\Theta_a$ with the anti-diagonal; see Figure~\ref{fig:two_max}. Consequently, $G_q\vert_{\Theta_a}$ must also attain at least four local minima, so, altogether, $G_q\vert_{\Theta_a}$ is stationary at at least seven points of $\Theta_a$ depicted in Figure~\ref{fig:two_max}.
In what follows, we show that such behavior is impossible, even under the most intricate scenario. 

We employ a similar argument to the one used in the proof of inequality \eqref{eq:mainineq}, but carried out at a deeper level: we analyze $x\mapsto G_q(x,\theta_a(x))$ by repeatedly differentiating, factoring the expression, and isolating the nontrivial factor.
We repeat this process until the remaining expression becomes sufficiently simple to analyze.

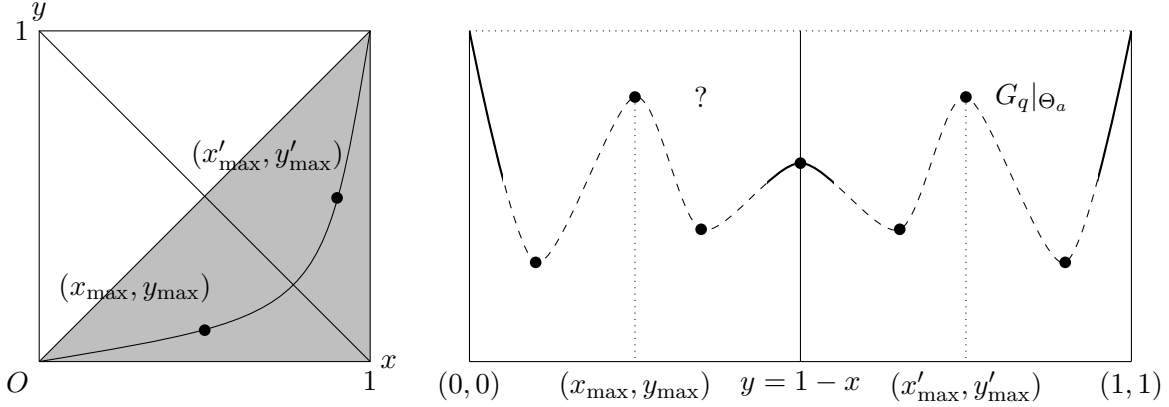
\begin{figure}
\centering

\begin{tikzpicture}[scale=4.4]
 \begin{scope}
\fill [lightgray,opacity=0.2] (0,0)--(1,0)--(1,1);
\draw (0,0)--(0,1)--(1,1)--(1,0)--(0,0);
\draw (0,0)--(1,1);
\draw (1,0)--(0,1);
\draw (0,0) .. controls (6/7,1/7).. (1,1);

\fill (1/2,0.095) circle (0.5pt);
\fill (0.9,0.495) circle (0.5pt);

\node [above left] at (0.55,0.15){$(x_{\max},y_{\max})$};
\node [above left] at (0.95,0.55) {$(x_{\max}',y_{\max}')$};

\node [below left]at(0,0) {$O$};
\node [below]at(1,0) {$1$};
\node [left]at(0,1) {$1$};

\node [right]at (1,0) {$x$};
\node [above]at (0,1) {$y$};

\end{scope}

\begin{scope}[shift={(1.3,0)}]

\draw (0,0)--(0,1);
\draw (0,0)--(2,0);
\draw (2,0)--(2,1);
\draw[dotted] (0,1)--(2,1);
\draw[ black, smooth, dashed] plot coordinates {
  (0,1) 
  (0.2,0.3) (0.5,0.8) (0.7,0.4) (1,0.6) (1.3,0.4) (1.5,0.8) (1.8,0.3) 
  (2,1)
};

\begin{scope}
\clip (0,0)--(0.1,0)--(0.1,1)--(0,1);
\draw[thick, black, smooth] plot coordinates {
  (0,1) 
  (0.2,0.3) (0.5,0.8) (0.7,0.4) (1,0.6) (1.3,0.4) (1.5,0.8) (1.8,0.3) 
  (2,1)
};
\end{scope}

\begin{scope}
\clip (2,0)--(1.9,0)--(1.9,1)--(2,1);
\draw[thick, black, smooth] plot coordinates {
  (0,1) 
  (0.2,0.3) (0.5,0.8) (0.7,0.4) (1,0.6) (1.3,0.4) (1.5,0.8) (1.8,0.3) 
  (2,1)
};
\end{scope}

\begin{scope}
\clip (1-0.1,0)--(1+0.1,0)--(1+0.1,1)--(1-0.1,1);
\draw[thick, black, smooth] plot coordinates {
  (0,1) 
  (0.2,0.3) (0.5,0.8) (0.7,0.4) (1,0.6) (1.3,0.4) (1.5,0.8) (1.8,0.3) 
  (2,1)
};
\end{scope}

\draw (1,0)--(1,1);

\draw [dotted](0.5,0.8)--(0.5,0);
\draw [dotted](1.5,0.8)--(1.5,0);

\fill (0.2,0.3) circle (0.5pt);
\fill (0.5,0.8) circle (0.5pt);
\fill (1,0.6) circle (0.5pt);
\fill (1.5,0.8) circle (0.5pt);
\fill (1.8,0.3) circle (0.5pt);
\fill (0.7,0.4) circle (0.5pt);
\fill (1.3,0.4) circle (0.5pt);

\node at (0.7,0.8) {$?$};
\node at (1.7,0.8) {$G_q\vert_{\Theta_a}$};

\node [below]at (0,0) {$(0,0)$};
\node [below]at (2,0) {$(1,1)$};
\node [below]at (1,0) {$y=1-x$};
\node [below]at (0.5,0) {$(x_{\max},y_{\max})$};
\node [below]at (1.5,0) {$(x_{\max}',y_{\max}')$};

\end{scope}
\end{tikzpicture}

\caption{The curve $\Theta_a$ passes through the hypothetical maxima of $G_q$ in $\Delta$ (left). The graph of $G_q$ evaluated along the curve $\Theta_a$ (right).}
\label{fig:two_max}
\end{figure}

Along the curve $\Theta_a$,
\[ G_q(x,\theta_a(x))
= (1-x+ax)^{-q/p}\Big( (1-x)^{2q/p} + \big(1+a^{1/p}\big)^q (1-x)^{q/p}x^{q/p} +a^{q/p} x^{2q/p} \Big) \]
and therefore 
\[ (1-x+ax)^{2q/p} \frac{\textup{d}}{\textup{d} x} G_q(x,\theta_a(x)) \]
under the parametrization $x=1/(u+1)$ becomes
\[ \frac{q}{p} (u+1)^{-3q/p+1} (u+a)^{q/p-1} \phi_0(u), \]
where $\phi_0\colon(0,\infty)\to\R$ is defined as
\begin{align*}
\phi_0(u)
&= -(1+a)u^{2q/p}-2a u^{2q/p-1}+\big(1+a^{1/p}\big)^{q}u^{q/p+1}\\
&\quad\, -a\big(1+a^{1/p}\big)^qu^{q/p-1}+2a^{q/p}u+(1+a)a^{q/p}.
\end{align*}

Observe the following chain of relations: 
\[ \phi_0''(u) = u^{q/p-3}\phi_1(u),\qquad 
\phi_1'(u) = u\phi_2(u). \]
Here, each function is explicitly given by
{\allowdisplaybreaks
\begin{align*}
p\phi_0'(u) 
& = - 2q(1+a)u^{2q/p-1} - 2(2q-p)a u^{2q/p-2} \\
&\quad\, + (q+p)\big(1+a^{1/p}\big)^q u^{q/p} 
- (q-p)a\big(1+a^{1/p}\big)^q u^{q/p-2}
+ 2pa^{q/p},\\
p^2\phi_1(u)
& =
- 2q(2q-p)(1+a)u^{q/p+1}
- 4(2q-p)(q-p)a u^{q/p}\\
&\quad\,+ q(q+p)\big(1+a^{1/p}\big)^q u^{2}
- (q-p) (q-2p) a\big(1+a^{1/p}\big)^q,\\
p^3\phi_2(u)
&= - 2q(2q-p)(q+p)(1+a) u^{q/p-1}
- 4q(2q-p)(q-p)a u^{q/p-2}\\
&\quad\,+ 2pq(q+p)\big(1+a^{1/p}\big)^q.
\end{align*}
}
Finally, we obtain
\[ p^4u^{3-q/p}\phi_2'(u) 
= - 2q(2q-p)(q+p)(q-p)(1+a)u -4(2q-p)(q-p)q(q-2p)a, \]
i.e.,
\begin{equation}\label{eq:phi2prime}
\frac{p^4u^{3-q/p}\phi_2'(u)}{2q(q-p)(2q-p)} = - (q+p) (1+a) u - 2 (q-2p) a.
\end{equation}

Recall that we have concluded that the function $G_q\vert_{\Theta_a}$ has at least seven stationary points on the curve $\Theta_a$, so $\phi_0$ has at least seven zeros in $(0,\infty)$.
By Rolle's theorem, $\phi'_0$ has at least six zeros, while $\phi''_0$, and thus also $\phi_1$, have at least five zeros in $(0,\infty)$. 
Next, $\phi'_1$, and thus also $\phi_2$, have at least four zeros in $(0,\infty)$.
By one last application of Rolle's theorem we deduce that $\phi'_2$ has at least three zeros, but this is not possible since the right hand side of \eqref{eq:phi2prime} is a first-degree polynomial.

We arrived at a contradiction with the assumption that $G_q$ attains a local maximum in $\Delta$.
We conclude that $G_q$ attains its maximum on $[0,1]^2$ either at the boundary or on the anti-diagonal $y=1-x$.
Lemma \ref{lm:boundary antidiagonal} then completes the proof of \eqref{eq:mainineqYoung}.

\begin{remark}\label{rem:thePDE}
In order for our method to work, it is crucial to find the ``correct'' curves for the dissection of $\Delta$. The authors were led to consider the candidate $y=\theta_a(x)$ upon observing the following PDE satisfied by $G_q$. 
If we denote 
{\allowdisplaybreaks
\begin{align*}
v(x,y) & := ((1-x)x, (1-y)y), \\
\tilde{a}(x,y) & := p^2x^{1/p}(1-y)^{1/p}, \\
\tilde{b}(x,y) & := p (1-x) x \Big( x^{1/p} (1-y)^{1/p} \big( p (1-2 x)+q (2x+2y-1) \big) \\
& \qquad\qquad\qquad - (1-x)^{1/p} y^{1/p} \big( p(1-2 x)+q (2x+2y-3) \big) \Big), \\
\tilde{c}(x,y) &: = q x^{1/p} (1-y)^{1/p} \Big( p \big((1-x)x+(1-y)y\big)+q \big((x+y)^2-x-3y\big) \Big),
\end{align*}
}
then
\begin{align*}
& \big(\tilde{a}(x,y)-\tilde{a}(y,x)\big) v(x,y) \nabla^2 G_q(x,y) v(x,y)^\textup{T}\\
& + \big(\tilde{b}(x,y),-\tilde{b}(y,x)\big)\cdot \nabla G_q(x,y) + \big(\tilde{c}(x,y)-\tilde{c}(y,x)\big) G_q(x,y)  = 0
\end{align*}
for $(x,y)\in(0,1)^2$. The directional Hessian term, $v(x,y) \nabla^2G_q(x,y) v(x,y)^\textup{T}$, represents the second derivative of $G_q$ along the vector field $v(x,y)$, which is associated with the ODE 
\[
\frac{\textup{d} y}{\textup{d} x}=\frac{(1-y)y}{(1-x)x}.
\]
The solutions to this ODE are, in fact, precisely the functions $\theta_a$, i.e., its integral curves are $\Theta_a$ from \eqref{e:dissecting curve}.
\end{remark}

\begin{remark}\label{rem:notamax}
If one wanted to characterize completely general Young's inequality \eqref{eq:toogeneralYoung} for exponents $1\leq p,q,r<\infty$, then one would like to know that the left hand side $\|f\ast g\|_{\ell^r}$ still attains maximum $1$ at
\[ (f,g)=(\mathbbm{1}_{\{0,1\}^d},\mathbbm{1}_{\{0,1\}^d}), \]
with the right hand side normalized by $\|f\|_{\ell^p} = \|g\|_{\ell^q} = 1$.
In one dimension this is the same as saying that $H_{p,q,r}\colon[0,1]^2\to(0,\infty)$ defined as
\[ H_{p,q,r}(x,y) := (1-x)^{r/p} (1-y)^{r/q} + \big( (1-x)^{1/p} y^{1/q} + x^{1/p} (1-y)^{1/q} \big)^{r} + x^{r/p} y^{r/q} \]
attains maximum $1$ at the central point $(x,y)=(1/2,1/2)$.
Let us explain that this is not always the case, even when $r=2$. Namely,
\[ H_{p,q,2}\Big(\frac{1}{2},\frac{1}{2}\Big) = \frac{6}{2^{2/p+2/q}}, \]
so $H_{p,q,2}(1/2,1/2)=1$ whenever the exponents $1\leq p,q<\infty$ satisfy
\begin{equation}\label{eq:pqlog6}
\frac{1}{p} + \frac{1}{q} = \frac{1}{2}\log_2 6 = 1.29248\ldots
\end{equation}
and one might expect the above to hold for all such choices of $p$ and $q$. 
However, despite the fact that $(1/2,1/2)$ is a stationary point of $H_{p,q,2}$,
\[ \nabla H_{p,q,2}\Big(\frac{1}{2},\frac{1}{2}\Big) = (0,0), \]
we also compute
\[ \nabla^2 H_{p,q,2}\Big(\frac{1}{2},\frac{1}{2}\Big) 
= 2^{4-2/p-2/q} \begin{pmatrix}
(4-3p)/p^2 & 0 \\
0 & (4-3q)/q^2
\end{pmatrix} \]
and notice that $\nabla^2 H_{p,q,2}(1/2,1/2)$ cannot be negative semi-definite, and $(1/2,1/2)$ cannot be a point of local maximum for $H_{p,q,2}$, as soon as $p<4/3$ or $q<4/3$.

Conversely, if $p,q\geq4/3$ satisfy \eqref{eq:pqlog6}, then it is possible to prove that \eqref{eq:toogeneralYoung} is actually true. Namely, in the endpoint case $p=4/3$ one can study the very concrete function $H_{4/3,4/(\log_2 9-1),2}$ using standard numerical tools (as in the proof of Lemma \ref{lm:cosh}), showing that it is at most $1$ on the whole square $[0,1]^2$. The other endpoint $q=4/3$ is analogous. In those endpoint cases one can then deduce \eqref{eq:toogeneralYoung} by the same induction on the dimension as in the previous section. Finally, Young's convolution inequality in the claimed range of exponents then follows by multilinear interpolation \cite[Chapter 3]{Thi06}. 
However, note that this still does not fully characterize all exponents $p,q$ for which \eqref{eq:toogeneralYoung} holds, even when we simply take $r=2$, but rather only those $p,q$ that additionally satisfy \eqref{eq:pqlog6}.
\end{remark}


\section{Proof of Corollary~\ref{cor:entropy2}}

If we substitute
\[ q = 1 + \varepsilon \]
for some $\varepsilon>0$, then the formula \eqref{eq:pqeqYoung} easily gives
\begin{align}
p & = \frac{2 + 2\varepsilon}{\log_2(2^{1+\varepsilon}+2)}
= \frac{1+\varepsilon}{1+\varepsilon/4+o(\varepsilon)} \nonumber \\
& = 1 + \frac{3}{4} \varepsilon + o(\varepsilon) \quad \text{as } \varepsilon\to0^+. \label{eq:asymofp2}
\end{align}
Just as in Section \ref{sec:entropy}, we easily obtain
\[ f(x)^p \stackrel{\eqref{eq:asymofp2}}{=} f(x) + \frac{3}{4} \varepsilon f(x) \ln f(x) + o(\varepsilon) \quad \text{as } \varepsilon\to0^+ \]
for each $x\in\Z^d$ such that $f(x)>0$, which leads to
\[ \|f\|_{\ell^p(\Z^d)} = 1 - \frac{3\ln 2}{4} \varepsilon\, \textup{H}_{\Z^d}(f) + o(\varepsilon) \quad \text{as } \varepsilon\to0^+. \]
The same formula holds for $g$, and we also have
\[ \|f\ast g\|_{\ell^q(\Z^d)} = 1 - (\ln 2) \varepsilon\, \textup{H}_{\Z^d}(f\ast g) + o(\varepsilon) \quad \text{as } \varepsilon\to0^+. \]
Therefore, from inequality \eqref{eq:improvedYoung} we obtain
\[ 1 - (\ln 2) \varepsilon\, \textup{H}_{\Z^d}(f\ast g) \leq 1 - \frac{3\ln 2}{4} \varepsilon\, \textup{H}_{\Z^d}(f) - \frac{3\ln 2}{4} \varepsilon\, \textup{H}_{\Z^d}(g) + o(\varepsilon), \]
so subtracting $1$, dividing by $\varepsilon\ln 2$, and letting $\varepsilon\to0^+$ we really deduce \eqref{eq:entropy34}, as desired.

Finally, if $f=g=2^{-d}\mathbbm{1}_{\{0,1\}^d}$, then
\[ (f\ast g)(x_1,x_2,\ldots,x_d) = 2^{|\{j\,:\,x_j=1\}|-2d} \]
for $(x_1,\ldots,x_d)\in\{0,1,2\}^d$ and $(f\ast g)(x_1,\ldots,x_d)=0$ otherwise.
For every integer $0\leq k\leq d$ we see that $f\ast g$ attains the value $2^{(d-k)-2d}$ for $\binom{d}{k}2^k$ tuples $(x_1,\ldots,x_d)$, so its entropy is
\begin{align*}
\textup{H}_{\Z^d}(f\ast g) & = - \sum_{k=0}^{d} \binom{d}{k}2^k 2^{-d-k} \log_2 2^{-d-k}
= 2^{-d} \sum_{k=0}^{d} \binom{d}{k} (d+k) \\
& = 2^{-d} d \sum_{k=0}^{d} \binom{d}{k} d + 2^{-d} \sum_{k=1}^{d} \frac{d}{k} \binom{d-1}{k-1} k
= 2^{-d} d \big( 2^d + 2^{d-1} \big) = \frac{3}{2} d.
\end{align*}
Since 
\[ \textup{H}_{\Z^d}(f) = \textup{H}_{\Z^d}(g) = - 2^d 2^{-d} \log_2 2^{-d} = d, \]
we see that equality holds in \eqref{eq:entropy34}.


\section*{Acknowledgment}
This work was supported by the Croatian Science Foundation under the project number HRZZ-IP-2022-10-5116 (FANAP).


\bibliography{dimension_free_binary}{}
\bibliographystyle{plainurl}

\end{document}